\documentclass[10pt]{article}
\usepackage{amsmath,amssymb,amscd,color}
\usepackage[unicode,breaklinks=true,colorlinks=true]{hyperref}
\newcommand{\bq}{\begin{equation}}
\newcommand{\eq}{\end{equation}}

\newcommand{\R}{{ \mathbb{R}  }}

\newcommand{\bbr}{{ \mathbb{R}  }}

\newcommand{\bke}[1]{\left( #1 \right)}

\newcommand{\bket}[1]{\left\{ #1 \right\}}
\newcommand{\norm}[1]{\left\Vert #1 \right\Vert}

\begin{document}
\bibliographystyle{plain}


\newtheorem{defn}{Definition}
\newtheorem{lemma}[defn]{Lemma}
\newtheorem{proposition}{Proposition}
\newtheorem{theorem}[defn]{Theorem}
\newtheorem{cor}{Corollary}
\newtheorem{remark}{Remark}
\numberwithin{equation}{section}

\def\Xint#1{\mathchoice
   {\XXint\displaystyle\textstyle{#1}}%
   {\XXint\textstyle\scriptstyle{#1}}%
   {\XXint\scriptstyle\scriptscriptstyle{#1}}%
   {\XXint\scriptscriptstyle\scriptscriptstyle{#1}}%
   \!\int}
\def\XXint#1#2#3{{\setbox0=\hbox{$#1{#2#3}{\int}$}
     \vcenter{\hbox{$#2#3$}}\kern-.5\wd0}}
\def\ddashint{\Xint=}
\def\dashint{\Xint-}
\def\aint{\Xint\diagup}

\newenvironment{proof}{{\bf Proof.}}{\hfill\fbox{}\par\vspace{.2cm}}
\newenvironment{pfthm1}{{\par\noindent\bf
            Proof of Theorem \ref{Theorem1} }}{\hfill\fbox{}\par\vspace{.2cm}}
\newenvironment{pfprop1}{{\par\noindent\bf
            Proof of Proposition  \ref{time-decay} }}{\hfill\fbox{}\par\vspace{.2cm}}
\newenvironment{pfthm3}{{\par\noindent\bf
            Proof of Theorem  \ref{Theorem2} }}{\hfill\fbox{}\par\vspace{.2cm}}
\newenvironment{pfthm-1}{{\par\noindent\bf
            Proof of Theorem  \ref{Prop-1} }}{\hfill\fbox{}\par\vspace{.2cm}}
\newenvironment{pfthm-2}{{\par\noindent\bf
            Proof of Theorem  \ref{Theorem-weaksolution} }}{\hfill\fbox{}\par\vspace{.2cm}}

\newenvironment{pfthm4}{{\par\noindent\bf
Sketch of proof of Theorem \ref{Theorem6}.
}}{\hfill\fbox{}\par\vspace{.2cm}}
\newenvironment{pfthm5}{{\par\noindent\bf
Proof of Theorem 5. }}{\hfill\fbox{}\par\vspace{.2cm}}
\newenvironment{pflemsregular}{{\par\noindent\bf
            Proof of Lemma \ref{sregular}. }}{\hfill\fbox{}\par\vspace{.2cm}}

\title{Local well-posedness in the Wasserstein space for a chemotaxis model coupled to Navier-Stokes equations
}
\author{Kyungkeun Kang and Haw Kil Kim}

\date{}

\maketitle
\begin{abstract}
We consider a coupled system of Keller-Segel type equations and the incompressible Navier-Stokes equations in spatial dimension two and
three. In the previous work \cite{KK}, we established the existence of a weak solution of a
Fokker-Plank equation in the Wasserstein space using the optimal transportation technique. Exploiting this result, we
constructed solutions of Keller-Segel-Navier-Stokes equations such that the density of biological organism belongs to the absolutely continuous
curves in the Wasserstein space.
In this work, we refine the result on the existence of a weak solution of a Fokker-Plank equation in the Wasserstein space. As a result, we
construct solutions of Keller-Segel-Navier-Stokes equations under weaker assumptions on the initial data.

\end{abstract}



\section{Introduction}
 \setcounter{equation}{0}

In this paper, we study an aerotaxis model formulating the dynamics
of oxygen, swimming bacteria, and viscous incompressible fluids in
$Q_T:=\bbr^d \times [0, T)$, $d=2,3$, $T>0$.
\begin{equation}\label{KSNS-May6-10}
\partial_t \rho + u \cdot \nabla  \rho - \Delta \rho= -\nabla\cdot (\chi (c) \rho \nabla c),
\end{equation}
\begin{equation}\label{KSNS-May6-20}
\partial_t c + u \cdot \nabla c-\Delta c =-k(c) \rho,
\end{equation}
\begin{equation}\label{KSNS-May6-30}
\partial_t u + u\cdot \nabla u -\Delta u +\nabla p=-\rho \nabla
\phi,\quad {\rm{div}\,} u=0.
\end{equation}

Here $\rho(t,\,x) : Q_{T} \rightarrow \R^{+}$, $c(t,\,x) : Q_{T}
\rightarrow \R^{+}$, $u(t,\, x) : Q_{T} \rightarrow \R^{d}$ and
$p(t,x) :  Q_{T} \rightarrow \R$ denote the biological cell
concentration, oxygen concentration, fluid velocity, and scalar
pressure, respectively, where $\R^+$ indicates the set of
non-negative real numbers. The oxygen consumption rate $k(c)$ and
the aerobic sensitivity $\chi (c)$ are nonnegative as functions of
$c$, namely $k, \chi:\R^+\rightarrow\R^+$ such that $k(c)=k(c(x,t))$
and $\chi(c)=\chi(c(x,t))$ and the time-independent function $\phi
=\phi (x)$ denotes the potential function, e.g., the gravitational
force or centrifugal force. Initial data are given by $(\rho_0(x),
c_0(x), u_0(x))$ with $\rho_0(x),\, c_0(x) \geq 0$ and $\nabla \cdot
u_0=0$. Tuval et al. proposed in \cite{TCDWKG}, describing behaviors
of swimming bacteria, {\it Bacillus subtilis} (see also
\cite{CFKLM}).

The above system \eqref{KSNS-May6-10}-\eqref{KSNS-May6-30} seems to
have similarities to the classical Keller-Segel model suggested by
Patlak\cite{Patlak} and Keller-Segel\cite{KS1, KS2}, which is given
as
\begin{equation}\label{KS-nD} \,\,\left\{
 \begin{array}{c}
 \rho_t=\Delta \rho-\nabla \cdot(\chi \rho\nabla c),\\
 \vspace{-3mm}\\
  c_t=\Delta c-\alpha c+\beta \rho,\\
 \end{array}
 \right.
\end{equation}
where $\rho=\rho(t,x)$ is the biological cell density and $c=c(t,x)$ is the
concentration of chemical attractant substance. Here, $\chi$ is the
chemotatic sensitivity, and $\alpha\ge 0$ and $\beta\ge 0$ are the
decay and production rate of the chemical, respectively. The system
\eqref{KS-nD} has been comprehensively studied and results are not
listed here (see e.g. \cite{Her-Vela, JL, NSY, OY, Win} and the
survey papers \cite{Horstmann1, Horstmann2}).

We remark that in the case that the effect of fluids is absent,
i.e., $u=0$ and $\phi=0$, the system
\eqref{KSNS-May6-10}-\eqref{KSNS-May6-30} becomes a Keller-Segel
type model with the negative term $-k(c) \rho$. It is due to the
fact that the oxygen concentration is consumed, while the chemical
substance is produced by $\rho$ in the Keller-Segel system
\eqref{KS-nD}.

Our main objective is to establish the existence of solution $\rho$
for the system \eqref{KSNS-May6-10}-\eqref{KSNS-May6-30} in the
Wasserstein space, which will be described later in detail.

We review some known results related to well-posedness of solutions
for the system \eqref{KSNS-May6-10}-\eqref{KSNS-May6-30}. Existence
of local-in-time solutions was proven for bounded domains in $\R^3$.
It was shown in \cite{DLM} that smooth solutions exist globally in
time, provided that initial data are very near constant steady
states and $\chi(\cdot)$, $k(\cdot)$ hold the following conditions:
\begin{equation}\label{Assumption1}
\chi'(\cdot)\ge 0,\quad k'(\cdot) >0,\quad
\left(\frac{k(\cdot)}{\chi(\cdot)} \right)^{''} <0.
\end{equation}
Global well-posedness of regular solutions was proved in \cite{Win2}
for large initial data for bounded domains in $\R^2$ with boundary
conditions $\partial_{\nu}\rho=\partial_{\nu} c=u=0$ under a similar
conditions as \eqref{Assumption1} on $\chi(\cdot)$ and $k(\cdot)$:
\begin{equation}\label{Kang-Kim-May6-90}
\left( \frac{k(\cdot)}{\chi(\cdot)} \right)^{'} >0, \quad
(\chi(\cdot) k(\cdot))' \ge 0, \quad \left(
\frac{k(\cdot)}{\chi(\cdot)} \right) ^{''} \le 0.
\end{equation}
Different structure conditions are given in \cite{AKY}, where \eqref{Kang-Kim-May6-90} is slightly relaxed, in case that $\chi$ is a constant.

In \cite[Theorem 1.1]{ckl} the first author et al. constructed
unique regular solutions for general $\chi$ and $\kappa$ in
following function classes:
\begin{equation}\label{Kang-Kim-May6-100}
(\rho, c, u) \in C\left([0,T^{\ast}); H^{m}(\mathbb{R}^{3})\right)
\times C\left([0,T^{\ast}); H^{m+1}(\mathbb{R}^{3})\right)\times
C\left([0,T^{\ast}); H^{m+1}(\mathbb{R}^{3})\right),
\end{equation}
in case that $\phi\in H^{m}$ and initial data belong to
\begin{equation}\label{Kang-Kim-May6-110}
(\rho_{0}, c_{0}, u_{0})\in H^{m}(\mathbb{R}^{3}) \times
H^{m+1}(\mathbb{R}^{3}) \times H^{m+1}(\mathbb{R}^{3}).
\end{equation}
In addition, if $\chi(\cdot)$ and $k(\cdot)$ satisfy the following
conditions, motivated by experimental results in \cite{CFKLM} and
\cite{TCDWKG} (compare to \eqref{Assumption1} or
\eqref{Kang-Kim-May6-90}): There is an $\epsilon>0$ such that
\begin{equation} \label{CKL10-march14}
\chi(c),\, k(c),\, \chi'(c),\, k'(c)\geq 0,\,\mbox {and }\, \sup |
\chi (c)- \mu k(c)| < \epsilon\,\,\mbox{ for some }\,\mu>0,
\end{equation}
global-in-time regular solutions exist in $\R^2$ with no smallness
of the initial data (see \cite[Theorem 1.3]{ckl}).
We remark that was shown  in \cite{ckl-nonlinearity} that if $\norm{\rho_0}_{L^1(\R^2)}$ is sufficiently small, global well-posedness and temporal decays
can be established,  in only case that $\chi(c),\, k(c),\, \chi'(c),\, k'(c)\geq 0$.
One can also consult \cite{ckl-cpde}, \cite{Win3}
and \cite{ckl-jkms} with reference therein for temporal
decay and asymptotics, and also refer to e.g.
\cite{CKK}, \cite{FLM}, \cite{Tao-W} and \cite{CK-jmp}  for the
nonlinear diffusion models of a porous medium type.

Our main ingredient is to establish existence of weak solutions for the Fokker-Plank equations in the Wasserstein space.
Compared to previous result, in \cite{KK}, the authors assumed that $\rho_0$ is in
$L^{1}(\R^d)\cap L^{\infty}(\R^d)$ and constructed
cell concentration $\rho$ in the Wasserstein space. More precisely, by understanding the following Fokker-Planck equation as an absolutly continuous curve
in the Wasserstein space,
\begin{equation} \label{KK-May7-10}
\left\{\begin{matrix}
   & \partial_t \rho =   \nabla \cdot( \nabla \rho  - v\rho)\\
   & \rho(\cdot,0)=\rho_0
    \end{matrix}
    \right.\quad\mbox{ in }\,\, Q_{T} := [0,\, T] \times \R^{d},
\end{equation}
we solved \eqref{KK-May7-10} under the assumption (refer Theorem 1.1 in \cite{KK})
\begin{equation}\label{precondition-v}
v\in L^2(0,T;L^\infty(\bbr^d)) \qquad \mbox{and} \qquad  \mbox{div}~ v \in L^1(0,T;L^\infty(\bbr^d)).
\end{equation}
As a result, by exploiting the above result with $v:=u-\chi(c)\nabla c$, we constructed the solution of Keller-Segel-Navier-Stokes system
\eqref{KSNS-May6-10}-\eqref{KSNS-May6-30} under the assumption (refer Theorem 1.2 in \cite{KK})
\begin{equation}
\rho_0 \in L^1(\bbr^d) \cap L^\infty(\bbr^d),\quad c_0 \in W^{3,m}(\bbr^d)\cap W^{2,2}(\bbr^d),\quad u_0 \in W^{2,b}(\bbr^d),
\end{equation}
for any $m>d $ and $b>d+2$.

In this paper, by exploiting approximation argument to \eqref{KK-May7-10} and being able to control the uniform speed of
approximated absolutely continuous curves  in the Wasserstein space, we solve the Fokker-Plank equation \eqref{KK-May7-10}
under the weaker regularity assumption on the velocity field
\begin{equation}
v\in L^{\beta}(0,T; L^{\alpha}(\bbr^d)), \qquad \frac{d}{\alpha} +\frac{2}{\beta} \leq 1,\,\,\alpha >d.
\end{equation}
This is certainly weaker than \eqref{precondition-v} and it turns out that initial condition $\rho_0$ can be in a function space larger  than
$L^1\cap L^{\infty}(\R^d)$.

More precise statement of the above result is stated in Theorem \ref{Prop-1}.

\begin{theorem}\label{Prop-1}
Let $\frac{d}{p}+\frac{2}{q}=1$ and $p>d$.  Suppose
\begin{equation}\label{KK-May7-30}
\rho_0 dx\in\mathcal{P}_2(\mathbb{R}^d),\qquad \int_{\bbr^d}\rho_0\ln \rho_0 dx < \infty.
\end{equation}
Assume further that
\begin{equation}\label{KK-May19-10}
v \in  L^q(0,T; L^p(\bbr^d)).
\end{equation}

 Then, there exists an absolutely continuous curve $\mu\in
AC_2(0,T; \mathcal{P}_2(\mathbb{R}^d))$ such that $\mu(t):=\rho(t,x) dx \in \mathcal{P}_2^{ac}(\bbr^d)$ for all $t\in [0,T]$, and $\rho$
solves \eqref{KK-May7-10} in the sense of
distributions, namely, for any $\varphi\in
C_c^\infty([0,T)\times \bbr^d)$
\begin{equation}\label{KK-May7-40}
\begin{aligned}
\int_{0}^{T} \int_{\bbr^d}  \left\{\partial_t\varphi(t,x)
 + \Delta \varphi(t,x)+ \nabla \varphi(t,x) \cdot v(t,x)\right\}\rho(t,x)dxdt =  - \int_{\bbr^d}\varphi(0,x) \rho_0(x)dx,
\end{aligned}
\end{equation}
and
\begin{equation}\label{KK-May7-60}
W_2(\rho(t),\rho(s))\leq C \sqrt{t-s}\qquad \forall ~~0\leq s\leq
t\leq T,
\end{equation}
where the constant $C=C\left (T, ~~ \int_{\bbr^d}\rho_0 \ln \rho_0 dx, ~~ \int_{\bbr^d}|x|^2 \rho_0 dx, ~~ \int_0^T  \|v_t
\|_{L^p}^{\frac{2p}{p-d}} dt \right )$

Furthermore, if $\rho_0 \in L^\alpha(\bbr^d) $ for $\alpha >1$ then we have
\begin{equation}\label{KK-May7-50}
||\rho(t)||_{L^\alpha(\bbr^d)} \leq \| \rho_0\|_{L^\alpha}
~\exp\Big(C\|v \|_{L^q(0,T;L^p(\bbr^d))}^{q} \Big)
\quad \forall ~ t\in[0,T].
\end{equation}

\end{theorem}

\begin{remark}
 The limiting case, $(p, q)=(d, \infty)$, in
\eqref{KK-May19-10} can be included, and it requres, however, an
extra smallness of its norm, i.e. there is an $\epsilon>0$ such that
$\norm{v}_{L^{\infty}(0,T; L^p(\bbr^d))}<\epsilon$. Thus, we do
not consider such case in Theorem \ref{Prop-1}.
\end{remark}

With the aid of Theorem \ref{Prop-1},  we construct weak
solutions for the aerotaxis-fluid model
\eqref{KSNS-May6-10}-\eqref{KSNS-May6-30} in the Wasserstein space. For
convenience, we introduce some function classes, which are defined as
\begin{equation}\label{KK-May7-120}
X_a((0,t)\times \R^d):=L^{\infty}(0,t; L^a(\R^d)),
\end{equation}
\begin{equation}\label{KK-May7-130}
Y_a((0,t)\times \R^d):=L^{2}(0,t; W^{2,a}(\R^d))\cap W^{1,2}(0,t;
L^{a}(\R^d)),
\end{equation}
where the function spaces $X_a$ and $Y_a$ are equipped
with the following norms:
\begin{equation}\label{KK-May23-10}
\norm{f}_{X_a}:=\norm{f}_{L^{\infty}((0,T); L^a(\R^d))},
\end{equation}
\begin{equation}\label{KK-May23-20}
\norm{f}_{Y_a}:=\norm{f}_{L^{2}((0,T); W^{2,
a}(\R^d))}+\norm{f}_{W^{1,2}((0,T); L^{a}(\R^d))}.
\end{equation}

Next, we establish the existence of solutions in the classes $(\rho, c, u)\in X_a\times Y_a \times Y_a$ for the system
\eqref{KSNS-May6-10}-\eqref{KSNS-May6-30}. Our result reads as follows:

\begin{theorem}\label{Theorem-weaksolution}
Let $d=2,3$, and $d/2<a<\infty$. Suppose that $\chi, k, \chi', k'$
are all non-negative and $\chi$, $k\in C^j(\R^+)$ and $k(0)=0$, $\|
\nabla \phi \|_{L^1\cap L^{\infty}}<\infty$ for $1<j<\infty$. Let
the initial data $(\rho_0, c_0, u_0)$ be given as
\begin{equation}\label{kk-April12-10}
\rho_0\in (L^1\cap L^{a})(\R^d), \qquad c_0, u_0\in W^{2,a}(\R^d).
\end{equation}
Then there exists $T>0$ and a unique weak solution $(\rho, c, u)$ of
\eqref{KSNS-May6-10}-\eqref{KSNS-May6-30} such that
\begin{equation}\label{kk-April12-20}
\rho\in X_a(Q_T) ,\qquad c, u\in Y_a(Q_T),
\end{equation}
where $Q_T=[0,T]\times \R^d$.
Furthermore, we have
\begin{equation}\label{KK-May27-300}
W_2\Big(\frac{\rho(t)}{\|\rho_0\|_{L^1(\bbr^d)}},\frac{\rho(s)}{\|\rho_0\|_{L^1(\bbr^d)}}\Big)\leq C \sqrt{t-s}\qquad \forall ~~0\leq s\leq
t\leq T.
\end{equation}
Here the constant $C$ in \eqref{KK-May27-300} depends on
\begin{equation}\label{KK-May27-400}
||(|u|+|\nabla c|)||_{ L^{\beta}(0,T; L^{\alpha}(\bbr^d))},
\quad  \int_{\bbr^d}|x|^2 \rho_0
dx, \quad \| \rho_0 \|_{L^a(\bbr^d)},
\end{equation}
where $\alpha>d$ and $\beta \geq 2$ are numbers that come from Theorem 1 when $p=a$.
\end{theorem}

We make several comments for Theorem \ref{Theorem-weaksolution}.
\begin{itemize}

\item[(i)] It was shown in  \cite{KMS16} that if the initial data  are sufficiently small in invariant classes, it is known that mild solutions globally exist in time.
More precisely, in case $\kappa(c)=c$ and under the assumption that
$\max\bket{
\norm{\rho_0}_{L^{\frac{d}{2}}}, \norm{c_0}_{(L^{\infty}\cap \dot{W}^{1,d}_{w})}, \norm{u_0}_{L^{d}_{w}}, \norm{\nabla\phi}_{L^{d}_{w}}}$
is small (if $d\ge 3$, $\norm{\rho_0}_{L^{\frac{d}{2}}}$ is relaxed by $\norm{\rho_0}_{L^{\frac{d}{2}}_w}$), mild solutions exist globally in time, e.g. $t^{\frac{d}{2}\bke{\frac{2}{d}-\frac{1}{q}}}\rho \in BC_{w}([0, \infty); L^q(\R^d))$ (see \cite[Theorem 1]{KMS16}).
The initial data in Theorem \ref{Theorem-weaksolution} are assumed to be subcritical, and thus stronger than those in \cite{KMS16}.
Nevertheless, the weak solutions constructed in Theorem \ref{Theorem-weaksolution} is for the case of large data, not small data and furthermore, the estimate \eqref {KK-May27-300} is valid continuously up to initial time (compare to \cite{KMS16}).

\item[(ii)] It is not clear if local well-posedness can be established in case that $\rho_0\in L^{\frac{d}{2}}(\R^d)$ with no smallness in $L^{\frac{d}{2}}(\R^d)$-norm or not,
 We also do not know that local well-posedness can be extened to global well-posedness, when $\rho_0\in L^{p}(\R^d)$ for $p>\frac{d}{2}$, and thus
we leave these as open questions.

\item[(iii)] The results of Theorem \ref{Theorem-weaksolution} also hold for $d>3$. We do not, however, include the case for the system \eqref{KSNS-May6-10}-\eqref{KSNS-May6-30}, since a dimension higer than three doesn't seem to be empirically relavant.

\item[(iv)] We note that, due to regularized effect of diffusion, solutions in Theorem \ref{Theorem-weaksolution}
become regular in $[\delta, T]\times \R^d$ for all $\delta>0$, assuming additionally that $\chi$ are $\kappa$ belong to $C^{\infty}(\R^+)$.

\item[(vi)] The initial data $c_0$ and $u_0$ can be relaxed compared to assumptions in \eqref{kk-April12-10}. To be more precise, $W^{2,a}(\R^d)$ in
\eqref{kk-April12-10} can be replaced by $D_a^{1-\frac{1}{a}, a}(\R^d)$ defined by
\[
D_a^{1-\frac{1}{a}, a}(\R^d):=\{ h\in L^a(\R^d): \norm{h}_{D_a^{1-\frac{1}{a}, a}}=\norm{h}_{L^a(\R^d)}
+(\int_0^{\infty} \norm{t^{\frac{1}{a}}A_ae^{-tA_a}h }^a_{L^a}\frac{dt}{t})^{\frac{1}{a} } <\infty
\},
\]
where $A_a$ is the Heat or Stokes operator (see e.g. \cite[Theorem 2.3]{GS}).

\end{itemize}

This paper is organized as follows. In Section 2, preliminary works
are introduced. Section 3 and Section 4 are devoted to proving Theorem \ref{Prop-1} and Theorem \ref{Theorem-weaksolution}, respectively.


\section{Preliminaries}

\subsection{Wasserstein space}
In this subsection, we introduce the Wasserstein space and remind some properties of it. For more detail, readers may refer to e.g. \cite{ags:book} and \cite{V}.

\begin{defn}
Let $\mu$ be a probability measure on $\mathbb{R}^d$. Suppose there is a measurable map $T: \mathbb{R}^d\mapsto \mathbb{R}^d$. Then, the map $T$ induces
a probability measure $ \nu$ on $\mathbb{R}^d$ which is defined as
$$\int_{\mathbb{R}^d} \varphi(y)d\nu(y) = \int_{\mathbb{R}^d} \varphi(T(x))d\mu(x) \qquad \forall ~ \varphi\in C(\mathbb{R}^d).$$
We denote, for convenience, $\nu:=T_\#\mu$ and say that $\nu$ is the push-forward of $\mu$ by $T$.
\end{defn}

Let us denote by $\mathcal{P}_2(\mathbb{R}^d)$ the set of all Borel
probability measures on $\mathbb{R}^d$ with a finite second moment. For $\mu,\nu\in\mathcal{P}_2(\mathbb{R}^d)$, we consider
\begin{equation}\label{Wasserstein dist}
W_2(\mu,\nu):=\left(\inf_{\gamma\in\Gamma(\mu,\nu)}\int_{\mathbb{R}^d\times
\mathbb{R}^d}|x-y|^2d\gamma(x,y)\right)^{\frac{1}{2}},
\end{equation}
where $\Gamma(\mu,\nu)$ denotes the set of all Borel probability
measures on $\mathbb{R}^d\times \mathbb{R}^d$ which has $\mu$ and
$\nu$ as marginals, i.e.
$$\gamma(A\times \mathbb{R}^d)= \mu(A)\qquad \gamma(\mathbb{R}^d\times A)= \nu(A) $$
for every Borel set $A\subset \mathbb{R}^d.$

Equation (\ref{Wasserstein dist}) defines a distance on
$\mathcal{P}_2(\mathbb{R}^d)$ which is called the {\it Wasserstein distance}.
 Equipped with the Wasserstein distance,  $\mathcal{P}_2(\mathbb{R}^d)$ is called the {\it Wasserstein space}.
 It is known that the infimum in the right hand side of Equation
(\ref{Wasserstein dist}) always achieved. We will denote by $\Gamma_o(\mu,\nu)$ the set of all $\gamma$ which minimize the expression.

If $\mu$ is
absolutely continuous with respect to the Lebesgue measure then there exists a convex function $\phi$ such that
$\gamma:=(Id\times \nabla \phi)_\# \mu$ is the unique element of $\Gamma_o(\mu,\nu)$, that is, $\Gamma_o(\mu,\nu)=\{(Id\times \nabla \phi)_\# \mu \}$.

\begin{defn}\label{Convexity}
Let $\phi:\mathcal{P}_2(\mathbb{R}^d)\mapsto (-\infty,\infty]$.
We say that $\phi$ is convex in $\mathcal{P}_2(\mathbb{R}^d)$ if for every
couple $\mu_1,\mu_2\in \mathcal{P}_2(\mathbb{R}^d)$ there exists an
optimal plan $\gamma\in\Gamma_o(\mu_1,\mu_2)$ such that
$$\phi(\mu_t^{1\rightarrow 2})\leq (1-t)\phi(\mu_1)+t\phi(\mu_2)\qquad \forall ~ t\in[0,1], $$
where $\mu_t^{1\rightarrow 2}$ is a constant speed geodesic
between $\mu^1$ and $\mu^2$ defined as
$$\mu_t^{1\rightarrow 2} := ((1-t)\pi^1 + t\pi^2)_\#\gamma .$$
Here, $\pi^1: \mathbb{R}^d\times \mathbb{R}^d \mapsto
\mathbb{R}^d$ and $\pi^2: \mathbb{R}^d\times \mathbb{R}^d \mapsto
\mathbb{R}^d$ are the first and second projections of $
\mathbb{R}^d\times \mathbb{R}^d$ onto $ \mathbb{R}^d$ defined by
$$\pi^1(x,y)=x,\quad  \pi^2(x,y)=y\qquad \forall ~ x,y \in \mathbb{R}^d.$$
\end{defn}

As we have seen in Theorem \ref{Prop-1} and Theorem \ref{Theorem-weaksolution}, we will find a solution of $\rho$ equation \eqref{KSNS-May6-10} in the class of absolutely continuous
curves in the Wasserstein space. Now, we introduce the definition of absolutely continuous curve and its relation with the continuity equation.
\begin{defn}
Let $\sigma:[a,b]\mapsto \mathcal{P}_2(\mathbb{R}^d)$ be a curve.
We say that $\sigma$ is absolutely continuous and denote it by $\sigma \in AC_2(a,b;\mathcal{P}_2(\mathbb{R}^d))$, if there exists $m\in
L^2([a,b])$ such that
\begin{equation}\label{AC-curve}
W_2(\sigma(s),\sigma(t))\leq \int_s^t m(r)dr\qquad \forall ~ a\leq s\leq t\leq b.
\end{equation}
If $\sigma \in AC_2(a,b;\mathcal{P}_2(\mathbb{R}^d))$, then the limit
$$|\sigma'|(t):=\lim_{s\rightarrow t}\frac{W_2(\sigma(s),\sigma(t))}{|s-t|} ,$$
exists for $L^1$-a.e $t\in[a,b]$. Moreover, the function $|\sigma'|$ belongs to $L^2(a,b)$ and satisfies
\begin{equation}
|\sigma'|(t)\leq m(t) \qquad \mbox{for} ~L^1-\mbox{a.e.}~t\in [a,b],
\end{equation}
for any $m$ satisfying \eqref{AC-curve}.
We call $|\sigma'|$ by the metric derivative of $\sigma$.
\end{defn}

\begin{lemma}[\cite{ags:book}, Theorem 8.3.1]\label{representation of AC curves}
If $\sigma\in AC_2(a,b;\mathcal{P}_2(\mathbb{R}^d))$ then there exists a Borel vector field $v: \bbr^d\times(a,b)\mapsto \bbr^d$ such that
$$v_t \in L^2(\sigma_t) \mbox{   for} ~~L^1-a.e  ~~ t\in [a,b]  , $$
and the continuity equation
\begin{equation}\label{continuity-eqn}
\partial_t\sigma_t +\nabla\cdot(v_t\sigma_t)=0,
\end{equation}
holds in the sense of distribution sense.

Conversely, if a weak* continuous curve $\sigma : [a,b] \mapsto \mathcal{P}_2(\mathbb{R}^d)$ satisfies the continuity equation \eqref{continuity-eqn} for some Borel vector field
$v_t$ with $||v_t||_{L^2(\sigma_t)}\in L^2(a,b)$, then $\sigma: [a,b]\mapsto \mathcal{P}_2(\mathbb{R}^d)$ is absolutely continuous and
$|\sigma_t'|\leq ||v_t||_{L^2(\sigma_t)}$ for $L^1$-a.e $t\in [a,b]$.
\end{lemma}
{\it Notation} : In Lemma \ref{representation of AC
curves}, we use notation $v_t:=v(\cdot,t)$ and
$\sigma_t:=\sigma(t)$. Throughout this paper, we keep this
convention, unless any confusion is to be expected, and a usual
notation $\partial_t$ is adopted for temporal derivative, i.e.
$f_t:=f(\cdot,t)$ and $\partial_t f:= \frac{\partial f}{\partial t}$.

\begin{lemma}\label{Lemma : Arzela-Ascoli}
Let $\sigma_n \in AC_2(a,b;\mathcal{P}_2(\mathbb{R}^d))$ be a
sequence and suppose there exists $m\in L^2([a,b])$ such that
\begin{equation}\label{equi-continuity}
W_2(\sigma_n(s),\sigma_n(t))\leq \int_s^t m(r)dr\qquad \forall ~
a\leq s\leq t\leq b,
\end{equation}
for all $n\in \mathbb{N}.$ Then there exists a subsequnece
$\sigma_{n_k}$ such that
\begin{equation}\label{eq1 : Lemma : Arzela-Ascoli}
\sigma_{n_k}(t) \mbox{ weak* converges to} ~~\sigma(t), \qquad
\text{for all} ~~t\in[a,b],
\end{equation}
for some $\sigma\in AC_2(a,b;\mathcal{P}_2(\mathbb{R}^d))$
satisfying
\begin{equation}\label{eq2 : Lemma : Arzela-Ascoli}
W_2(\sigma(s),\sigma(t))\leq \int_s^t m(r)dr\qquad \forall ~ a\leq
s\leq t \leq b.
\end{equation}
\end{lemma}
\begin{proof}
Refer proposition 3.3.1 in \cite{ags:book} with the fact that $\mathcal{P}_2(\mathbb{R}^d)$ is weak* compact.
\end{proof}


\subsection{Estimates of heat equation and Stokes system}

We first recall some estimates of the heat equation, which are useful for our purpose.
For convenience, we denote $Q_t=\R^d\times [0,t]$ for $t>0$.

Let $w$ be the solution of the following heat equation:
\begin{equation}\label{heat-equ}
\partial_t w-\Delta w =\nabla\cdot g+h \quad\mbox{ in } Q_t,\qquad w(x,0)=w_0,
\end{equation}
where $g$ is a d-dimensional vector field and $h$ is a scalar function.

Let $(\alpha, \beta)$ with $d/\alpha+2/\beta\le 1$ and $d<\alpha<\infty$.
Suppose that $\nabla w_0\in L^{\alpha}_x$, $g\in L^{\beta}_tL^{\alpha}_x$ and $h\in L^{\beta}_tL^{\gamma}_x$, where $1<\gamma\le \alpha$
with $1/r<1/\alpha+1/d$. Then, it follows that
\begin{equation}\label{heat-est10}
\norm{\nabla w}_{L^{\beta}_tL^{\alpha}_x}\le C\bke{\norm{g}_{L^{\beta}_tL^{\alpha}_x}
+t^{\frac{1}{2}(\frac{d}{\alpha}+1-\frac{d}{\gamma})}\norm{h}_{L^{\beta}_tL^{\gamma}_x}
+t^{\frac{1}{\beta}}\norm{\nabla w_0}_{L^{\alpha}_x}}
\end{equation}
The estimate \eqref{heat-est10} is well-known, but, for clarity, we give a sketch of its proof.

Indeed, we decompose $w=w_1+w_2+w_3$ such that $w_i$, $i=1,2,3$ satisfies
\[
\partial_t w_1-\Delta w_1 =\nabla\cdot g \quad\mbox{ in } Q_t,\qquad w_1(x,0)=0
\]
\[
\partial_t w_2-\Delta w_2 =h \quad\mbox{ in } Q_t,\qquad w_2(x,0)=0
\]
\[
\partial_t w_3-\Delta w_3 =0 \quad\mbox{ in } Q_t,\qquad w_3(x,0)=w_0
\]
Via representation formula of the heat equation  for example $w_1$,
\[
w_1(x, t)=-\int_0^{t}\int_{\R^d}\nabla\Gamma(x-y,t-s)g(y,s)dyds,
\]
where $\Gamma$ is the heat kernel.
Due to maximal regularity, we observe that
\[
\norm{\nabla w_1}_{L^{\beta}_tL^{\alpha}_x}\le C\norm{g}_{L^{\beta}_tL^{\alpha}_x}.
\]
On the other hand, it is easy to see that
\[
\norm{\nabla w_3}_{L^{\beta}_tL^{\alpha}_x}\le Ct^{\frac{1}{\beta}}\norm{\nabla w_0}_{L^{\alpha}_x}.
\]
We suffices to show that
\begin{equation}\label{heat-h10}
\norm{\nabla w_2}_{L^{\beta}_tL^{\alpha}_x}\le Ct^{1-\frac{d}{2\gamma}}\norm{h}_{L^{\beta}_tL^{\gamma}_x}.
\end{equation}
Indeed, via representation formula of the heat equation, we have
\[
w_2(x, t)=\int_0^{t}
\int_{\R^d}\Gamma(x-y,t-s)h(y,s)dyds.
\]
Using potential estimate, we have
\[
\norm{\nabla w_2(t)}_{L^{\alpha}_x}\le C\int_0^t (t-s)^{-\frac{d}{2}(\frac{1}{\gamma}-\frac{1}{\alpha})-\frac{1}{2}}\norm{h(s)}_{L^{\gamma}_x}ds.
\]
Integrating in time, we obtain
\[
\norm{\nabla w_2}_{L^{\beta}_tL^{\alpha}_x}\le C t^{-\frac{d}{2}(\frac{1}{\gamma}-\frac{1}{\alpha})+\frac{1}{2}}
\norm{h}_{L^{\beta}_tL^{\gamma}_x}.
\]
We deduce the estimate \eqref{heat-est10}.

\begin{remark}
In case that $g=0$, for any $(\alpha, \beta)$ with $d/\alpha+2/\beta\le 1$, $d<\alpha\le \infty$ and $\gamma<\infty$
\begin{equation}\label{heat-est15}
\norm{\nabla w}_{L^{\beta}_tL^{\alpha}_x}\le C\bke{t^{\frac{1}{2}(\frac{d}{\alpha}+1-\frac{d}{\gamma})}\norm{h}_{L^{\beta}_tL^{\gamma}_x}
+t^{\frac{1}{\beta}}\norm{\nabla w_0}_{L^{\alpha}_x}}.
\end{equation}
Compared to \eqref{heat-est10}, the case that $(\alpha, \beta)=(\infty, 2)$ is also included in \eqref{heat-est15}.
\end{remark}

We also remind the maximal regularity of heat equation and Stoke system (see e.g. \cite{LSU} and
\cite{S}). Let $w$ be a solution of
\[
w_t-\Delta w=f\qquad \mbox{in }\,\,Q_T:=\R^d\times [0,T]
\]
\[
w(x,0)=w_0(x)\qquad \mbox{in }\,\,\R^d.
\]
Then, the following a priori estimate holds for any $1<p, q<\infty$
\begin{equation}\label{KK-June-1000}
\norm{w_t}_{L^{p,q}_{x,t}(Q_T)}+\norm{\nabla^2
w}_{L^{p,q}_{x,t}(Q_T)}\le
C\norm{f}_{L^{p,q}_{x,t}(Q_T)}+CT^{\frac{1}{q}}\norm{\nabla^2 w_0}_{L^{p}(\R^d)}.
\end{equation}
In case of the following Stokes
system
\[
v_t-\Delta v+\nabla \pi=f,\quad{\rm div}\,\,v=0\qquad \mbox{in
}\,\,Q_T:=\R^d\times [0,T]
\]
\[
v(x,0)=v_0(x)\qquad \mbox{in }\,\,\R^d.
\]
Similarly, it is known that the following a priori estimate holds:
\begin{equation}\label{KK-June-10000}
\norm{v_t}_{L^{p,q}_{x,t}(Q_T)}+\norm{\nabla^2
v}_{L^{p,q}_{x,t}(Q_T)}+\norm{\nabla
\pi}_{L^{p,q}_{x,t}(Q_T)}\le
C\norm{f}_{L^{p,q}_{x,t}(Q_T)}+CT^{\frac{1}{q}}\norm{\nabla^2 v_0}_{L^{p}(\R^d)},
\end{equation}
where $1<p,q<\infty$.

In next lemma, we obtain some estimates of functions in $Y_a$, which we will use later.

\begin{lemma}
Let $d/2<a\le d$ and $d/\alpha+2/\beta= 1$. Assume that $\alpha, \beta, p$ and $q$ are numbers satisfying
\[
\frac{d}{\alpha}+\frac{2}{\beta}=1, \quad d<\alpha<\infty,\qquad \frac{d}{p}+\frac{2}{q}=\frac{d}{a}.
\]
Suppose that $f\in Y_a(Q_T)$ and $f(x,0)=f_0\in (L^{\alpha}\cap L^{\infty})(\R^d)$. Then,
\begin{equation}\label{Feb29-90-a}
\norm{f}_{L^2_tL^{\infty}_x(Q_T)}\le
CT^{1-\frac{d}{2a}}\norm{f}_{Y_a(Q_T)}  +CT^{\frac{1}{2}}\norm{f_0}_{L^{\infty}(\R^d)},
\end{equation}
\begin{equation}\label{Nov2-10}
\norm{f}_{L^{\beta}_tL^{\alpha}_x(Q_T)}\le
CT^{1-\frac{d}{2a}}\norm{f}_{Y_a(Q_T)}  +CT^{\frac{1}{\beta}}\norm{f_0}_{L^{\alpha}(\R^d)}.
\end{equation}

\end{lemma}
\begin{proof}
Indeed, we consider
\[
\partial_t f-\Delta f=g\quad \mbox{ in }\,\,Q_t,\qquad
f(x,0)=f_0(x)\quad \mbox{ in }\,\,\R^d.
\]
Via representation formula of the heat equation, we have
\[
f(x, t)=\int_{\R^d}\Gamma(x-y,t)f_0(y)dy+\int_0^{t}
\int_{\R^d}\Gamma(x-y,t-s)g(y,s)dyds,
\]
where $\Gamma$ is the heat kernel. Since $g\in L^2_tL^a_x(Q_T)$, direct computations show that
\[
\norm{f}_{L^{\beta}_tL^{\alpha}_x(Q_T)}\le
C\norm{\int_0^T (t-s)^{-\frac{d}{2}(\frac{1}{a}-\frac{1}{2a})}\norm{g(s)}_{L^a_x}ds}_{L^{\beta}_t}+Ct^{\frac{1}{\beta}}\norm{f_0}_{L^{\alpha}_x}
\]
\[
\le
CT^{\frac{1}{2}(1-\frac{d}{2a})}\norm{g}_{L^2_tL^a_x}+Ct^{\frac{1}{\beta}}\norm{f_0}_{L^{\alpha}_x}
\]
\[
\le
CT^{\frac{1}{2}(1-\frac{d}{2a})}\norm{f}_{Y_a(Q_t)}+Ct^{\frac{1}{\beta}}\norm{f_0}_{L^{\alpha}_x}.
\]
Since the other estimates can be
similarly verified, we skip its details.
\end{proof}


\section{Proof of Theorem \ref{Prop-1}}

In this section, we provide the proof of Theorem \ref{Prop-1}. We start with some a priori estimates.
\begin{lemma}\label{Lemma : a priori 1}
Let $\frac{d}{p}+\frac{2}{q}=1$ and $p>d$.  Suppose
 that a non-negative function $\rho_0:\R^d\rightarrow \R$ and a vector field
$v:\R^d\rightarrow \R^d$ satisfy
$$\rho_0 \in L^\alpha(\bbr^d) \cap C_c^\infty(\bbr^d) \qquad \mbox{and} \qquad v\in L^q(0,T; L^p(\bbr^d)) \cap C_c^\infty([0,T]\times \bbr^d).$$%
Let $\alpha >1$ and $\rho $ be a solution of
\begin{equation}\label{eq 130 : A priori}
\partial_t\rho =\nabla\cdot(\nabla \rho - v \rho),\qquad \rho(0,\cdot)=\rho_0.
\end{equation}
Then we have
\begin{equation}\label{eq 170 : A priori}
\sup_{t\in[0,T]}\|\rho(t)\|_{L^\alpha }^\alpha \leq \|\rho_0\|_{L^\alpha}^\alpha e^{C\int_0^T \|v(\tau)\|_{L^p}^q d\tau}.
\end{equation}
\end{lemma}

\begin{proof}
We multiply \eqref{eq 130 : A priori} with $\rho^{\alpha-1}$ and integrate w.r.t spacial variable to get
\[
\frac{1}{\alpha}\frac{d}{dt} \int_{\mathbb{R}^d}\rho^{\alpha}~dx+\frac{4(\alpha-1)}{\alpha^2}\int_{\mathbb{R}^d} \big| \nabla \rho^{\frac{\alpha}{2}}\big|^2  \,dx
=-\int_{\mathbb{R}^d}\nabla \cdot \big(v\rho \big)~\rho^{\alpha-1} \,dx.
\]
Due to the Gagliardo-Nirenberg inequality, namely
\begin{equation}\label{eq 210 : A priori}
\|g \|_\beta \le  C\|\nabla g \|_2^{1-\theta'}\|g\|_2^{\theta'}  \qquad \rm{for}
\quad \frac{1}{\beta}=\Big(\frac{1}{2} - \frac{1}{d}\Big)\theta' +\frac{1-\theta'}{2}, \quad 2<\beta<2^*=\frac{2d}{d-2},
\end{equation}
the righthand side is estimated as follows:
\[
-\int_{\mathbb{R}^d}\nabla \cdot \big(v\rho \big)~\rho^{\alpha-1}~dx
=(\alpha-1)\int_{\mathbb{R}^d} \rho^{\alpha-1}v\cdot \nabla \rho ~dx
=\frac{2(\alpha-1)}{\alpha} \int_{\mathbb{R}^d} \rho^{\frac{\alpha}{2}} v\cdot \nabla \rho^{\frac{\alpha}{2}} ~dx
\]
\begin{equation}\label{eq 220 : A priori}
\lesssim \| v\|_{L^p} \| \nabla \rho^{\frac{\alpha}{2}} \|_{L^2}
\| \rho^{\frac{\alpha}{2}}\|_{\frac{2 p}{p-2}}\le \frac{2(\alpha-1)}{\alpha^2}\| \nabla \rho^{\frac{\alpha}{2}} \|^2_{L^2}+C\| v\|_{p}^{\frac{2 p}{p-d}} \| \rho\|_{L^\alpha}^{\alpha}.
\end{equation}
Therefore, we obtain
\begin{equation}\label{eq 140 : A priori}
\begin{aligned}
\frac{1}{\alpha}\frac{d}{dt} \| \rho\|_{L^\alpha}^{\alpha} + \frac{2(\alpha-1)}{\alpha^2}\| \nabla \big(\rho^{\frac{\alpha}{2}} \big)\|_{L^2}^2
 \leq C\| v\|_{p}^{\frac{2 p}{p-d}} \| \rho\|_{L^\alpha}^{\alpha},
 \end{aligned}
\end{equation}
which yields
\begin{equation}\label{eq 150 : A priori}
\begin{aligned}
\| \rho(t)\|_{L^\alpha}^{\alpha} \leq \| \rho_0\|_{L^\alpha}^{\alpha}~e^{C\int_0^t \| v(\tau)\|_{p}^{\frac{2 p}{p-d}} d\tau}.
\end{aligned}
\end{equation}
This completes the proof.
\end{proof}
Next, we estimate the Wasserstein distance, which turns out to be H\"older continuous.

\begin{lemma}\label{Lemma : a priori 2}
Let $\frac{d}{p}+\frac{2}{q}=1$ and $p>d$.  Suppose
 that a non-negative function $\rho_0:\R^d\rightarrow \R$ and a vector field
$v:\R^d\rightarrow \R^d$ satisfy
$$\rho_0 \in \mathcal{P}_2(\bbr^d) \cap C_c^\infty(\bbr^d) \qquad \mbox{and} \qquad v\in L^q(0,T; L^p(\bbr^d)) \cap C_c^\infty([0,T]\times \bbr^d).$$
 Let 
$\rho \in AC_2(0,T;\mathcal{P}_2(\bbr^d))$ be a solution of
\begin{equation}\label{eq 1 : A priori}
\partial_t\rho =\nabla\cdot(\nabla \rho - v \rho),\qquad \rho(0,\cdot)=\rho_0.
\end{equation}
Then we have
\begin{equation}\label{eq 2 : A priori}
W_2(\rho_s,\rho_t) \leq  C \sqrt{t-s}, \qquad \mbox{for all} \quad 0\leq s<t \leq T
\end{equation}
for some  positive constant 
$C=C\bke{T,  \int_{\bbr^d}\rho_0 \ln \rho_0 dx, \int_{\bbr^d}|x|^2 \rho_0 dx, \| v\|_{L^{q}(0,T; L^p(\bbr^d)) }}$.
\end{lemma}

\begin{proof}
First, we estimate the second moment of $\rho$. We multiply $|x|^2$ to \eqref{eq 1 : A priori} and integrate
\begin{equation}\label{eq 4 : A priori}
\begin{aligned}
\frac{d}{dt} \int_{\bbr^d}\rho_t |x|^2 dx &= \int_{\bbr^d} |x|^2 \nabla \cdot(\nabla \rho_t-v_t \rho_t) dx
=2d \int_{\bbr^d} \rho_t dx-\int_{\bbr^d} x v_t\rho dx\\
&\leq C + \epsilon \mu \|\nabla \rho_t^{\frac{1}{2}} \|_{L^{2}}^{2} + C \|x \rho_t^{\frac{1}{2}} \|_{L^2}^2 + C \|v_t \|_{L^p}^{\frac{2p}{p-d}},
\end{aligned}
\end{equation}
where we used
\begin{equation}
\begin{aligned}
\int_{\bbr^d}|x v_t|\rho_t dx &\leq \|x \rho_t^{\frac{1}{2}} \|_{L^2} \|\rho_t^{\frac{1}{2}} \|_{L^{\frac{2p}{p-2}}} \|v_t \|_{L^p}
\leq C  \|x \rho_t^{\frac{1}{2}} \|_{L^2} \|\nabla \rho_t^{\frac{1}{2}} \|_{L^{2}}^{\frac{d}{p}} \|v_t \|_{L^p}\\
&\leq \epsilon \mu \|\nabla \rho_t^{\frac{1}{2}} \|_{L^{2}}^{2} + C \|x \rho_t^{\frac{1}{2}} \|_{L^2}^2 +C\|v_t \|_{L^p}^{\frac{2p}{p-d}}.
\end{aligned}
\end{equation}
Here $\mu \in (0,1)$ which will be specified later.

Next, we prove $\rho \in AC_2(0,T;\mathcal{P}_2(\bbr^d))$. To do this, we rewrite \eqref{eq 1 : A priori} as follows
$$\partial_t \rho +  \nabla \cdot (w \rho )=0, \qquad \mbox{where}\quad
w:= -\frac{\nabla \rho}{\rho}+ v .$$
We then show that
\begin{equation}\label{eq 5 : A priori}
\int_0^T ||w_t||_{L^2(\rho_t)}^2 dt <\infty.
\end{equation}
Once we obtain \eqref{eq 5 : A priori}, from Lemma \ref{representation of AC curves}, we are done with proving $\rho \in AC_2(0,T;\mathcal{P}_2(\bbr^d))$.
Thus, it suffices to prove \eqref{eq 5 : A priori}.
We multiply $\ln \rho$ to \eqref{eq 1 : A priori} and then integrate
\begin{equation}\label{eq 6 : A priori}
\begin{aligned}
&\frac{d}{dt}\int_{\bbr^d}\rho_t \ln \rho_t dx +\int_{\bbr^d}\frac{|\nabla \rho_t|^2}{\rho_t} dx
= \int_{\bbr^d}v_t \cdot \nabla \rho_t dx\\
&\leq \|v_t\|_{L^p} \| \nabla \rho_t^{\frac{1}{2}}\|_{L^2} \| \rho_t^{\frac{1}{2}}\|_{L^{\frac{2p}{p-2}}}
\leq c \|v_t\|_{L^p} \| \nabla \rho_t^{\frac{1}{2}}\|_{L^2}^{1+\frac{d}{p}}
\leq \epsilon  \| \nabla \rho_t^{\frac{1}{2}}\|_{L^2}^{2} + C \|v_t\|_{L^p}^{\frac{2p}{p-d}},
\end{aligned}
\end{equation}
where we used that
\begin{equation}
 \| \rho_t^{\frac{1}{2}}\|_{L^{\frac{2p}{p-2}}} \leq  \| \rho_t^{\frac{1}{2}}\|_{L^{2}}^{\frac{p-d}{p}} \| \nabla \rho_t^{\frac{1}{2}}\|_{L^2}^{\frac{d}{p}}.
\end{equation}
We note there exist constants $a,~b >0$(independent of $\rho$) such that
\begin{equation}\label{eq 9 : A priori}
-\int_{\rho(x)<1} \rho \ln \rho~ dx \leq \int_{\bbr^d}|x|^2 \rho(x) dx + a.
\end{equation}
Combining \eqref{eq 4 : A priori} and \eqref{eq 6 : A priori}, we have
\[
\frac{d}{dt} \Big(\mu \int_{\bbr^d} \rho_t \ln \rho_t dx + \int_{\bbr^d} |x|^2 \rho_t dx\Big)
+\mu(1-2\epsilon) \int_{\bbr^d}\frac{|\nabla \rho_t|^2}{\rho_t}
\]
\begin{equation}\label{eq 14 : A priori}
\leq C + C \int_{\bbr^d} |x|^2 \rho_t dx + C \|v_t \|_{L^p}^{\frac{2p}{p-d}}.
\end{equation}
On the other hand,
\begin{equation}\label{eq 15 : A priori}
\begin{aligned}
\int_{\bbr^d} |v_t|^2 \rho_t dx \leq \| v_t\|_{L^p}^2 \|\rho_t^{\frac{1}{2}}\|_{L^{\frac{2p}{p-2}}}^2
\leq \| v_t\|_{L^p}^2 \|\nabla \rho_t^{\frac{1}{2}}\|_{L^2}^{\frac{2d}{p}}
\leq C \| v_t\|_{L^p}^{\frac{2p}{p-d}} + \mu \epsilon |\nabla \rho_t^{\frac{1}{2}}\|_{L^2}^2
\end{aligned}
\end{equation}
From \eqref{eq 14 : A priori} and \eqref{eq 15 : A priori}, we get
\[
\frac{d}{dt} \Big(\mu \int_{\bbr^d} \rho_t \ln \rho_t dx + \int_{\bbr^d} |x|^2 \rho_t dx\Big)
+\mu(1-3\epsilon) \int_{\bbr^d}\frac{|\nabla \rho_t|^2}{\rho_t}
+\int_{\bbr^d} |v_t|^2 \rho dx
\]
\begin{equation}\label{eq 16 : A priori}
\leq C + C \int_{\bbr^d} |x|^2 \rho_t dx + C \|v_t \|_{L^p}^{\frac{2p}{p-d}}.
\end{equation}
Integrating in time, we have
\begin{equation*}
\begin{aligned}
\Big(\mu \int_{\bbr^d} \rho_t \ln \rho_t dx + \int_{\bbr^d} |x|^2 \rho_t dx\Big)
+\mu(1-3\epsilon) \int_0^t \int_{\bbr^d}\frac{|\nabla \rho_\tau|^2}{\rho_\tau} dx~ d\tau + \int_0^t \int_{\bbr^d}|v_\tau|^2 \rho_\tau dx~ d\tau\\
\leq \Big(\mu \int_{\bbr^d} \rho_0 \ln \rho_0 dx + \int_{\bbr^d} |x|^2 \rho_0 dx\Big)+  Ct + C \int_0^t\int_{\bbr^d} |x|^2 \rho_\tau dx ~d\tau
+ C \int_0^t \|v_\tau \|_{L^p}^{\frac{2p}{p-d}} d\tau.
\end{aligned}
\end{equation*}
We denote
$$
f(t)=\mu \int_{\bbr^d}\rho_t (\ln \rho_t)_+ dx + \int_{\bbr^d}|x|^2 \rho_t dx ,
$$
$$
g(t)=\mu(1-3\epsilon) \int_0^t \int_{\bbr^d}\frac{|\nabla \rho_\tau|^2}{\rho_\tau} dx~ d\tau + \int_0^t \int_{\bbr^d}|v_\tau|^2 \rho_\tau dx~ d\tau.
$$
 Then
\begin{equation}
\begin{aligned}
f(t) +g(t) &\leq f(0) +Ct +C \int_0^t \int_{\bbr^d} |x|^2 \rho_\tau dx~ d\tau +C\int_0^t \|v_\tau \|_{L^p}^{\frac{2p}{p-d}}~d\tau
+ \mu \int_{\bbr^d}\rho_t (\ln \rho_t)_-\\
&\leq f(0) +Ct +C \int_0^t f(\tau) d\tau  +C\int_0^t \|v_\tau \|_{L^p}^{\frac{2p}{p-d}} ~d\tau + \mu \int_{\bbr^d}\rho_t(\ln \rho_t)_-~dx\\
&\leq  f(0) +C(t+1) +C \int_0^t f(\tau) d\tau  +C\int_0^t \|v_\tau \|_{L^p}^{\frac{2p}{p-d}}d\tau + \mu \int_{\bbr^d}|x|^2\rho_t dx\\
&\leq  f(0) +C(t+1) +C \int_0^t f(\tau) d\tau  +C\int_0^t \|v_\tau \|_{L^p}^{\frac{2p}{p-d}} d\tau+ \mu f(t) .
\end{aligned}
\end{equation}
Taking $\mu$ sufficiently small, it follows that
\begin{equation}
f(t) +g(t) \leq  Cf(0) +C(t+1) +C \int_0^t f(\tau) d\tau  +C\int_0^t \|v \|_{L^p}^{\frac{2p}{p-d}}. 
\end{equation}
Gronwall's inequality implies that
\begin{equation}
f(t) +g(t) \leq \big(1+te^{Ct} \big)\Big( f(0) +C(t+1)+C \int_0^t  \|v_\tau \|_{L^p}^{\frac{2p}{p-d}}d\tau\Big).
\end{equation}
Therefore, we obtain
\begin{equation}
\int_{\bbr^d}\rho_t (\ln \rho_t)_+ dx + \int_{\bbr^d}|x|^2 \rho_t dx
+ \int_0^t \int_{\bbr^d}\frac{|\nabla \rho_\tau|^2}{\rho_\tau} dx~ d\tau + \int_0^t \int_{\bbr^d}|v_\tau|^2 \rho_\tau dx~ d\tau\leq C=C(t),
\end{equation}
where $C(t)$ depends on
$$t, ~~ \int_{\bbr^d}\rho_0 \ln \rho_0 dx, ~~ \int_{\bbr^d}|x|^2 \rho_0 dx, ~~ \int_0^t  \|v_\tau \|_{L^p}^{\frac{2p}{p-d}} d\tau .$$
It follows that
\begin{equation}
\int_{\bbr^d}\rho_t (\ln \rho_t)_- \leq C=C(t),
\end{equation}
which leads to
\begin{equation}\label{eq 17 : A priori}
\int_{\bbr^d}\rho_t| \ln \rho_t|dx + \int_{\bbr^d}|x|^2 \rho_t dx
 + \int_0^t \int_{\bbr^d} \Big (\frac{|\nabla \rho_\tau|^2}{\rho_\tau} + |v_\tau|^2 \rho_\tau \Big) dx d\tau\leq C=C(t).
\end{equation}
%
%
Finally, for any $t\in[0,T]$, we have
\begin{equation}\label{eq 11 : A priori}
\begin{aligned}
\int_0^t ||w_\tau||_{L^2(\rho_\tau)}^2 d\tau &\leq 2 \int_0^t \int_{\bbr^d}
\frac{|\nabla \rho_\tau|^2}{\rho_\tau} dx d\tau + 2 \int_0^t \int_{\bbr^d} |v_\tau|^2\rho_\tau dx d\tau \\
&\leq C(T).
\end{aligned}
\end{equation}
Due to \eqref{eq 17 : A priori}, the constant $C(T)$ in \eqref{eq 11 : A priori}
 only depends on
$$T, ~~ \int_{\bbr^d}\rho_0 \ln \rho_0 dx, ~~ \int_{\bbr^d}|x|^2 \rho_0 dx, ~~ \int_0^T  \|v_t \|_{L^p}^{\frac{2p}{p-d}} dt.$$
Hence, we have
\begin{equation}\label{eq 12 : A priori}
\begin{aligned}
W_2(\rho(s),\rho(t)) &\leq \int_s^t  ||w_\tau||_{L^2(\rho_\tau)}
d\tau \leq  C(T) \sqrt{t-s},
\end{aligned}
\end{equation}
which concludes \eqref{eq 2 : A priori}.
\end{proof}
Now we present the proof of Theorem \ref{Prop-1}.
\\
\begin{pfthm-1}
 Suppose $v \in  L^q(0,T; L^p(\bbr^d))$ and $\rho_0\in  \mathcal{P}_2(\bbr^d)$ with $\int_{\bbr^d} \rho_0 \ln \rho_0 dx < \infty$.
By exploiting truncation and mollification, we may choose a sequence of vector fields $v^n\in C_c^\infty([0,T]\times \bbr^d)) $
such that
\begin{equation}\label{eq7 : Theorem-1}
\lim_{n \rightarrow \infty}||v^n - v ||_{L^q(0,T; L^p(\bbr^d))} =0.
\end{equation}
Using truncation, mollification and normalization in a similar way, we choose a sequence of functions
$\rho_0^n\in C_c^\infty(\bbr^d)\cap \mathcal{P}_2(\bbr^d)$ satisfying
\begin{equation}\label{eq9 : Theorem-1}
\lim_{n \rightarrow \infty} W_2(\rho_0^n , \rho_0) =0,
\quad  \lim_{n \rightarrow \infty} \int_{\bbr^d} \rho_0^n \ln \rho_0^n dx = \int_{\bbr^d} \rho_0 \ln \rho_0 dx.
\end{equation}
From Theorem of \cite{KK}, we have $\rho^n \in AC_2(0,T; \mathcal{P}_2(\bbr^d))$ which
is a solution of
\begin{equation}\label{eq5 : Theorem-1}
\left\{\begin{matrix}
   & \partial_t \rho^n =   \nabla \cdot( \nabla \rho^n  - v^n \rho^n)\\
   & \rho^n(0,\cdot)=\rho_0^n,
    \end{matrix}
    \right.
\end{equation}
that is,
\begin{equation}\label{eq21 : Theorem-1}
\int_{0}^{T} \int_{\bbr^d}  \left\{\partial_t\varphi(t,x)
 + \Delta \varphi(t,x)+ \nabla \varphi(t,x) \cdot v^n (t,x)\right\}\rho^n(t,x)
 =  - \int_{\bbr^d}\varphi(0,x) \rho_0^n(x),
\end{equation}
for every $\varphi\in C_c^\infty([0,T)\times \bbr^d)$.

\noindent {\it Step 1.}  We claim that curves $\rho^n :[0,T]  \mapsto \mathcal{P}_2(\bbr^d)$ are equi-continuous.\\
First, 
we exploit Lemma \ref{Lemma : a priori 2} and get
\begin{equation}
W_2(\rho^n_s,\rho^n_t) \leq  C_n \sqrt{t-s}
\end{equation}
where $C_n$ only depending on
\begin{equation}
T, ~~ \int_{\bbr^d}\rho_0^n \ln \rho_0^n dx, ~~ \int_{\bbr^d}|x|^2 \rho_0^n dx, ~~ \int_0^T  \|v_t^n \|_{L^p}^{\frac{2p}{p-d}} dt
\end{equation}
Due to \eqref{eq7 : Theorem-1} and \eqref{eq9 : Theorem-1},  
we conclude
\begin{equation}\label{eq 16 : A priori}
W_2(\rho^n_s,\rho^n_t) \leq  C \sqrt{t-s} \qquad \mbox{for all sufficiently large} ~~n,
\end{equation}
where $C$ only depending on
\begin{equation}
T, ~~ \int_{\bbr^d}\rho_0 \ln \rho_0 dx, ~~ \int_{\bbr^d}|x|^2 \rho_0 dx, ~~ \int_0^T  \|v_t \|_{L^p}^{\frac{2p}{p-d}} dt
\end{equation}

\noindent {\it Step 2.} Estimation \eqref{eq 16 : A priori} says that curves $\rho^n :[0,T]  \mapsto \mathcal{P}_2(\bbr^d)$ are equi-continuous and hence there
exists a curve $\rho:[0,T] \mapsto \mathcal{P}_2(\bbr^d)$ such that , as $n \rightarrow \infty$(up to subsequence)
\begin{equation}\label{eq33 : Theorem-1}
\rho^n(t) ~~\mbox{weak* converges to}~~ \rho(t) \quad \mbox{and} \quad W_2(\rho(s),\rho(t)) \leq   C \sqrt{t-s}, \qquad \forall ~ 0\leq s\leq t\leq T.
\end{equation}
Due to \eqref{eq 17 : A priori}, we note that the uniform entropy bound on $\rho^n$ implies
$\rho(t) \in \mathcal{P}_2^{ac}(\bbr^d)$ for all $t\in[0,T]$. Furthermore, the convergence in \eqref{eq7 : Theorem-1}
gives us
\begin{equation}\label{eq32 : Theorem-1}
\begin{aligned}
\int_{0}^{T} \int_{\bbr^d}  (\nabla \varphi \cdot v^n ) \rho^n dxdt
\longrightarrow  \int_{0}^{T} \int_{\bbr^d}  (\nabla \varphi \cdot v) \rho dxdt, ~~~\mbox{as}~~ n\rightarrow \infty,
\end{aligned}
\end{equation}
for any $\varphi \in C_c^\infty([0,T)\times \bbr^d) $.
Indeed, we have
\begin{equation}
\begin{aligned}
&\int_{0}^{T} \int_{\bbr^d}  \big[\nabla \varphi \cdot v^n \rho^n
-  \nabla \varphi \cdot v \rho\big ] dxdt \\
& = \int_{0}^{T} \int_{\bbr^d}  \nabla \varphi \cdot \big (v^n- v\big )\rho^n  dxdt
+ \int_{0}^{T} \int_{\bbr^d}  \nabla \varphi \cdot  \bke{v\rho^n- v\rho} ~ dxdt \\
& = I + II,
\end{aligned}
\end{equation}
where
\begin{equation}\label{eq34 : Theorem-1}
\begin{aligned}
| I |&\leq \|\nabla \varphi \|_{L^\infty[0,T)\times \bbr^d} \int_{0}^{T} \int_{\bbr^d} |v^n - v | \rho^n ~dx dt\\
&\leq T\|\nabla \varphi \|_{L^\infty[0,T)\times \bbr^d}  \int_{0}^{T} \int_{\bbr^d} |v^n - v |^2 \rho^n ~dx  dt.
\end{aligned}
\end{equation}
Due to \eqref{eq 15 : A priori}, we note that
\begin{equation}\label{eq35 : Theorem-1}
\begin{aligned}
 \int_{0}^{T} \int_{\bbr^d} |v^n - v |^2 \rho^n ~dx  dt &\leq \int_0^T \|v^n-v\|_{L^p}^2 \| \nabla (\rho^n)^\frac{1}{2} \|_{L^2}^\frac{2d}{p} dt\\
 &\leq \Big (\int_0^T \|v^n-v\|_{L^p}^{\frac{2p}{p-d}} \Big)^{\frac{p-d}{p}} \Big( \int_0^T \| \nabla (\rho^n)^\frac{1}{2} \|_{L^2}^2\Big)^{\frac{d}{p}}\\
 &\leq C \|v^n-v \|_{L^{\frac{2p}{p-d}}(0,T ; L^p(\bbr^d))}^2
\end{aligned}
\end{equation}
where the last inequality follows from \eqref{eq 17 : A priori}. Combining \eqref{eq34 : Theorem-1} and \eqref{eq35 : Theorem-1}, we have
\begin{equation}
| I |\leq C \|v^n-v \|_{L^{\frac{2p}{p-d}}(0,T ; L^p(\bbr^d))}^2 \longrightarrow 0
\end{equation}
as $n\rightarrow \infty$ due to \eqref{eq7 : Theorem-1}. Also, weak* convergence in \eqref{eq33 : Theorem-1} implies that
$II$ converges to $0$ as $n \rightarrow 0$.
We plug \eqref{eq33 : Theorem-1} and \eqref{eq32 : Theorem-1} into \eqref{eq21 : Theorem-1}, and get
\begin{equation}\label{eq22 : Theorem-1}
\begin{aligned}
\int_{0}^{T} \int_{\bbr^d}  \left\{\partial_t\varphi(t,x)
 + \Delta \varphi(t,x)+ \nabla \varphi(t,x) \cdot v(t,x)\right\}\rho (t,x)dxdt =  - \int_{\bbr^d}\varphi(0,x) \rho_0(x)dx,
\end{aligned}
\end{equation}
for each $\varphi\in C_c^\infty([0,T)\times \bbr^d)$. This implies that $\rho$ solves
\begin{equation}\label{eq8 : Theorem-1}
\left\{\begin{matrix}
   & \partial_t \rho =   \nabla \cdot( \nabla \rho  - v\rho)\\
   & \rho(\cdot,0)=\rho_0.
    \end{matrix}
    \right.
\end{equation}

\noindent {\it Step 3.}  If $\rho_0 \in L^\alpha(\bbr^d)$ for $\alpha>1$ then, in addition to \eqref{eq9 : Theorem-1}, we may choose $\rho_0^n$ satisfying
\begin{equation}
\lim_{n\rightarrow \infty} \|\rho_0^n-\rho_0 \|_{L^\alpha(\bbr^d)} =0.
\end{equation}
From \eqref{eq 170 : A priori}, we have
\begin{equation}\label{eq10 : Theorem-1}
\sup_{t\in[0,T]}\|\rho^n(t)\|_{L^\alpha }^\alpha \leq \|\rho_0^n\|_{L^\alpha}^\alpha e^{C\int_0^T \|v^n(\tau)\|_{L^p}^q d\tau}.
\end{equation}
Due to the lower semicontinuity of $L^\alpha$-norm with respect to the weak* convergence, we take $n \rightarrow \infty$ in \eqref{eq10 : Theorem-1} and get
\begin{equation}
\sup_{t\in[0,T]}\|\rho(t)\|_{L^\alpha }^\alpha \leq \|\rho_0\|_{L^\alpha}^\alpha e^{C\int_0^T \|v(\tau)\|_{L^p}^q d\tau},
\end{equation}
which completes the proof.
\end{pfthm-1}


\section{Proof of Theorem \ref{Theorem-weaksolution} }

Reminding function spaces $X_a$ and  $Y_a$ defined in
\eqref{KK-May7-120}--\eqref{KK-May7-130}, we denote $W:=X_a\times
Y_a\times Y_a$. Let $M>0$ be a positive number and we introduce
$W^M:=X^M_a\times Y^M_a\times Y^M_a$, where
\[
X^M_a:=\{f\in X_a|\norm{f}_{X_a}\le M\},
\qquad
Y^M_a:=\{f\in Y_a|\norm{f}_{Y_a}\le M\}.
\]
For convenience, for $a>d/2$ we denote
\begin{equation}\label{KK-May24-100}
\eta:=1-\frac{d}{2a}=\frac{2a-d}{2a}.
\end{equation}
To construct solutions, we will use the method of iterations.
Setting $\rho_1(x,t):=\rho_0(x)$,  $c_1(x,t):=c_0(x)$ and $u_1(x,t):=u_0(x)$, we
consider the following: $k=1, 2,\cdots$
\begin{equation}\label{KK-Nov28-1}
\partial_t \rho_{k+1} =   \nabla
\cdot\bke{ \nabla \rho_{k+1}  -u_k\rho_{k+1}-\chi(c_k) \nabla
 c_{k+1}\rho_{k+1}},\qquad \rho_{k+1}(\cdot,0)=\rho_0,
\end{equation}
\begin{equation}\label{KK-Nov28-5}
\partial_t c_{k+1} -\Delta c_{k+1}=-u_k\nabla c_{k+1}  -\kappa(c_k) \rho_{k},\qquad
c_{k+1}(\cdot,0)=c_0,
\end{equation}
\begin{equation}\label{KK-Nov28-10}
\partial_t u_{k+1} -\Delta u_{k+1}+\nabla p_{k+1}=-u_k\nabla u_k  -
\rho_{k+1}\nabla\phi,\quad{{\rm div}\, u_{k+1}}=0\quad
u_{k+1}(\cdot,0)=u_0.
\end{equation}
Here for given $M>0$ we assume that $\rho_0$ and
$c_0$ are non-negative and
\begin{equation}\label{KK-march8-200}
\norm{\rho_0}_{(L^1\cap L^{a})(\R^d)}<\frac{M}{6},\qquad
\norm{c_0}_{W^{2,a}(\R^d)}<\frac{M}{6},\qquad
\norm{u_0}_{W^{2,a}(\R^d)}<\frac{M}{6}.
\end{equation}
We note, due the maximum principle, that $c_k$
is uniformly bounded, i.e. $\norm{c_k(\cdot,
t)}_{L^{\infty}(\R^d)}\le \norm{c_0}_{L^{\infty}(\R^d)}$.
Now we are ready to present the proof of Theorem \ref{Theorem-weaksolution}.
\\
\begin{pfthm-2}
In case that $a>d$ is rather easy. From now on, we consider only the case that $d/2<a\le d$.
For simplicity, we assume that $\int_{\bbr^d}\rho_0(x) dx
=1$ (if not, we replace $\rho_0$ by $\rho_0/\int_{\R^d}\rho_0 dx$, which doesn't yield any crucial change for the local existence of solutions). Under the hypothesis that $(\rho_k, c_k, u_k)\in W^M$, we first
show that $(\rho_{k+1}, c_{k+1}, u_{k+1})\in W^M$ for sufficiently
small $T$, which will be specified later.

First, recalling the equation \eqref{KK-Nov28-1} and using the estimate  \eqref{KK-May7-50},
in particular the case that
$(\alpha, \beta)=(2a, \tilde{a})$, i.e. $\tilde{a}=4a/(2a-d)$
we obtain for all $t\in[0,T]$
\[
||\rho_{k+1}(t)||_{L^a(\bbr^d)} \leq
||\rho_0||_{L^a(\bbr^d)}\exp\bke{C(\norm{u_k+\chi(c_k) \nabla c_{k+1}(\tau) }_{L^{\tilde{a}}_tL^{2a}_x })^{\frac{4a}{2a-d}}}
\]
\[
\leq ||\rho_0||_{L^a(\bbr^d)}\exp\bke{C(\norm{u_k}_{L^{\tilde{a}}_tL^{2a}_x }+
\norm{\nabla c_{k+1}}_{L^{\tilde{a}}_tL^{2a}_x })^{\frac{4a}{2a-d}} }.
\]
We note, due to \eqref{heat-est10}, that
\[
\norm{\nabla c_{k+1}}_{L^{\tilde{a}}_tL^{2a}_x }\le C\norm{u_kc_{k+1}}_{L^{\tilde{a}}_tL^{2a}_x }+Ct^{\frac{1}{2}\bke{1-\frac{d}{2a}}}\norm{\rho_k}_{L^{\tilde{a}}_tL^{a}_x }
+Ct^{\frac{2a-d}{4a}}\norm{\nabla c_0}_{L^{2a}_x }
\]
\[
\le C\norm{u_k}_{L^{\tilde{a}}_tL^{2a}_x }+Ct^{\frac{1}{2}\bke{1-\frac{d}{2a}}}\norm{\rho_k}_{L^{\tilde{a}}_tL^{a}_x }
+Ct^{\frac{2a-d}{4a}}\norm{\nabla c_0}_{L^{2a}_x }
\]
\begin{equation}\label{18Nov3-300}
\le Ct^{1-\frac{d}{2b}}\norm{u_k}_{Y_b(Q_t)} +Ct^{\frac{3}{2}\bke{1-\frac{d}{2a}}}\norm{\rho_k}_{L^{\infty}_tL^{a}_x }
+Ct^{\frac{1}{2}\bke{1-\frac{d}{2a}}}\bke{\norm{u_0}_{L^{2a}_x }+\norm{\nabla c_0}_{L^{2a}_x }},
\end{equation}
where we used \eqref{Nov2-10} and $\norm{c_{k+1}(t)}_{L^{\infty}}\le \norm{c_0}_{L^{\infty}}$.
Thus, summing up the estimates, we have
\[
\norm{\rho_{k+1}(t)}_{L^a(\bbr^d)} \leq ||\rho_0||_{L^a(\bbr^d)}\times
\]
\[
\times \exp\bke{CT^{\eta}\norm{u_k}_{Y_a(Q_t)} +CT^{\frac{3\eta}{2}}\norm{\rho_k}_{X_a(Q_t)}
+CT^{\frac{\eta}{2}}\bke{\norm{u_0}_{L^{2a}_x }+\norm{\nabla c_0}_{L^{2a}_x }}}.
\]
Since $\norm{u_k}_{Y_a(Q_t)}\le M$, $|| \rho_k ||_{X_a(Q_T)}\le M$ and
$||\rho_0||_{L^a(\bbr^d)}<M/2$, by taking a sufficiently
small $T>0$ we obtain
\begin{equation}\label{KK-May26-10}
||\rho_{k+1}||_{X_a(Q_T)}\le M.
\end{equation}

Next, we consider the equation \eqref{KK-Nov28-10}.
Let $p$ with $a<p<\tilde{a}=ad/(d-a)$ with $a>d/2$ and $q$ with $d/p+2/q=d/a$.
We then define numbers $\tilde{p}$ and $\tilde{q}$ by $1/\tilde{p}+1/p=1/a$ and
$1/\tilde{q}+1/q=1/2$, respectively. We note that $d/\tilde{p}+2/\tilde{q}=1$ and $\tilde{p}>d$.
We then compute that
\[
\norm{u_k\nabla u_k}_{L^2_tL^a_x(Q_T)}\le
\norm{\norm{u_k}_{L^{\tilde{p}}_x}\norm{\nabla
u_k}_{L^p_x}}_{L^2_t}\le
\norm{u_k}_{L^{\tilde{q}}_t L^{\tilde{p}}_x}\norm{\nabla
u_k}_{L^q_tL^p_x}
\]
\[
\le C\bke{T^{\eta}\norm{u_k}_{Y_a(Q_t)}  +T^{\frac{1}{\tilde{q}}}\norm{u_0}_{L^{\tilde{p}}(\R^d)}}
\bke{\norm{u_k}_{Y_a(Q_t)}  +T^{\frac{1}{q}}\norm{\nabla u_0}_{L^{p}(\R^d)}}
\]
\begin{equation}\label{KK-April20-10}
\le
C\bke{T^{\eta}\norm{u_k}_{Y_a(Q_t)}  +(T^{\eta+\frac{1}{q}}+T^{\frac{1}{\tilde{q}}})\norm{u_0}_{W^{2,a}(\R^d)} }
\norm{u_k}_{Y_a(Q_t)}  +CT^{\frac{1}{2}}\norm{u_0}^2_{W^{2,a}(\R^d)}.
\end{equation}
Via maximal regularity \eqref{KK-June-10000} of the Stokes system and
\eqref{KK-April20-10}, we have for all $k\geq 1$
\[
\norm{\partial_t u_{k+1}}_{L^2_tL^a_x(Q_T)}+\norm{\nabla^2
u_{k+1}}_{L^2_tL^a_x(Q_T)}+\norm{\nabla p_{k+1}}_{L^2_tL^a_x(Q_T)}
\]
\[
\le C\bke{\norm{u_k\nabla
u_k}_{L^2_tL^a_x(Q_T)}+\norm{\rho_{k+1}\nabla\phi}_{L^2_tL^a_x(Q_T)}+T^{\frac{1}{2}}\norm{u_0}_{W^{2,a}(\R^d)}}
\]
\[
\le
C\bke{T^{\eta}\norm{u_k}_{Y_a(Q_t)}  +(T^{\eta+\frac{1}{q}}+T^{\frac{1}{\tilde{q}}})\norm{u_0}_{W^{2,a}(\R^d)} }
\norm{u_k}_{Y_a(Q_t)}
\]
\begin{equation}\label{KK-Nov28-100}
+CT^{\frac{1}{2}}\norm{\rho_{k+1}}_{X_a(Q_T)} +C( T^{\frac{1}{2}}\norm{u_0}_{W^{2,a}(\R^d)}+T^{\frac{1}{2}})\norm{u_0}_{W^{2,a}(\R^d)},
\end{equation}
where we used that $\norm{\nabla\phi}_{L^{\infty}(\R^d)}<C$.

For the
control of lower order derivative of $u_{k+1}$, via potential
estimate, we can also have
\[
\norm{u_{k+1}}_{L^2_tL^a_x(Q_T)}\le
C\bke{(T^{\frac{1}{2}}+T)\norm{u_k}^2_{Y_a(Q_t)}  +T^{\frac{1}{2}}\norm{u_0}^2_{W^2_a(Q_t)}}
\]
\begin{equation}\label{KK-June-10}
+CT^{\frac{3}{2}}\norm{\rho_{k+1}}_{X_a(Q_T)}+CT^{\frac{1}{2}}\norm{u_0}_{L^{a}(\R^d)}.
\end{equation}
Thus, combining estimates \eqref{KK-May26-10}, \eqref{KK-Nov28-100} and
\eqref{KK-June-10}, and taking $T$ sufficiently small, we note that
\begin{equation}\label{KK-May26-20}
||u_{k+1}||_{Y_a(Q_T)}\le M.
\end{equation}
On the other hand, similarly as in \eqref{KK-April20-10}, we note
that
\[
\norm{u_k\nabla c_{k+1}}_{L^2_tL^a_x(Q_T)}\le
\norm{\norm{u_k}_{L^{\tilde{p}}_x}\norm{\nabla
c_{k+1}}_{L^p_x}}_{L^2_t}\le
\norm{u_k}_{L^{\tilde{q}}_xL^{\tilde{p}}_x}\norm{\nabla
c_{k+1}}_{L^q_tL^p_x}
\]
\[
\le C\bke{T^{\eta}\norm{u_k}_{Y_a(Q_t)}  +T^{\frac{1}{\tilde{q}}}\norm{u_0}_{L^{\tilde{p}}(\R^d)}}
\bke{\norm{c_{k+1}}_{Y_a(Q_t)}  +T^{\frac{1}{q}}\norm{\nabla c_0}_{L^{p}(\R^d)}}
\]
\[
\le
C\bke{T^{\eta}\norm{u_k}_{Y_a(Q_t)}  +T^{\frac{1}{\tilde{q}}}\norm{u_0}_{L^{\tilde{p}}(\R^d)}}
\norm{c_{k+1}}_{Y_a(Q_t)}
\]
\begin{equation}\label{KK-Nov3-10}
+CT^{\eta+\frac{1}{q}}\norm{u_k}_{Y_a(Q_t)}\norm{\nabla c_0}_{L^{p}(\R^d)}  +CT^{\frac{1}{2}}\norm{u_0}_{L^{\tilde{p}}(\R^d)}\norm{\nabla c_0}_{L^{p}(\R^d)}.
\end{equation}
Using  maximal regularity \eqref{KK-June-1000} for the equation \eqref{KK-Nov28-5} and the estimate
\eqref{KK-Nov3-10}, we have
\[
\norm{\partial_t c_{k+1}}_{L^2_tL^a_x(Q_T)}+\norm{\nabla^2
c_{k+1}}_{L^2_tL^a_x(Q_T)}
\]
\[
\le C\bke{\norm{u_k\nabla c_{k+1}}_{L^2_tL^a_x(Q_T)}
+\norm{\kappa(c_k) \rho_{k}}_{L^2_tL^a_x(Q_T)}+T^{\frac{1}{2}}\norm{c_0}_{W^{2,a}(\R^d)}}
\]
\[
\le
CT^{\frac{1}{2}}\bke{T^{\eta}\norm{u_k}_{Y_a(Q_t)}+\norm{u_0}_{W^{2,a}(\R^d)}}
\norm{c_{k+1}}_{Y_a(Q_t)}
+CT^{\frac{1}{2}+\eta}\norm{c_0}_{W^{2,a}(\R^d)}\norm{u_k}_{Y_a(Q_t)}
\]
\begin{equation}\label{KK-Nov28-101}
+CT^{\frac{1}{2}}\norm{\rho_{k}}_{X_a}+C( T^{\frac{1}{2}}\norm{u_0}_{W^{2,a}(\R^d)}+T^{\frac{1}{2}})\norm{c_0}_{W^{2,a}(\R^d)},
\end{equation}
where we used that
$\norm{\kappa(c_k)}_{L^{\infty}(\R^d)}<C=C(\norm{c_0}_{L^{\infty}})$
and $\eta$ is the number defined in \eqref{KK-May24-100}.

Since
estimates of lower order derivative of $c_{k+1}$ can be obtained as
in \eqref{KK-June-10}, we skip its details (from now on the estimate
of lower order derivatives are omitted, unless it is necessary to be
specified). Taking $T$ small enough such that
\[
T^{\frac{1}{2}}\bke{T^{\eta}\norm{u_k}_{Y_a(Q_t)}+\norm{u_0}_{W^{2,a}(\R^d)}}\le
CT^{\frac{1}{2}}\bke{T^{\eta}M+\norm{u_0}_{W^{2,a}(\R^d)}}\le
\frac{1}{2},
\]
the estimate \eqref{KK-Nov28-101} becomes
\begin{equation*}
\norm{c_{k+1}}_{Y_a}\le
 C\bke{T^{\frac{1}{2}+\eta}\norm{c_0}_{W^{2,a}(\R^d)}\norm{u_k}_{Y_a(Q_t)}
+T^{\frac{1}{2}}\norm{\rho_{k}}_{X_a}+(  T^{\frac{1}{2}}\norm{u_0}_{W^{2,a}(\R^d)}+ T^{\frac{1}{2}})\norm{c_0}_{W^{2,a}(\R^d)}},
\end{equation*}
which yields again for a sufficiently small $T$
\begin{equation}\label{KK-May26-30}
\norm{c_{k+1}}_{Y_a}\le {\color{red}M.}
\end{equation}

Next, we show that this iteration gives a fixed point via
contraction, which turns out to be unique solution of the system
under considerations. Let $\delta_k \rho:=\rho_{k+1}-\rho_k$,
$\delta_k c:=c_{k+1}-c_k$ and $\delta_k u:=u_{k+1}-u_k$. We then see
that $(\delta_k \rho, \delta_k c, \delta_k u)$ solves
\[
\partial_t \delta_k\rho -\Delta \delta_k\rho=   -\nabla
\cdot\Big( u_k \delta_k\rho+\delta_{k-1} u
\rho_k+\chi(c_k) \nabla  c_{k+1} \delta_k\rho
\]
\begin{equation}\label{KK-Nov28-300}
+\chi(c_k) \nabla\delta_{k} c\rho_k+\delta_{k-1}\chi(\cdot)
\nabla c_{k+1} \rho_k\Big),
\end{equation}
\begin{equation}\label{KK-Nov28-310}
\partial_t \delta_k c -\Delta \delta_k c=-\delta_{k-1} u\nabla c_{k+1}-u_k\nabla \delta_{k}c
-\delta_{k-1}\kappa(\cdot) \rho_{k}-\kappa(c_k) \delta_{k-1}\rho,
\end{equation}
\begin{equation}\label{KK-Nov28-320}
\partial_t \delta_k u -\Delta \delta_k u+\nabla \delta_k p=-\delta_{k-1} u\nabla u_k
- u_{k-1}\nabla \delta_{k-1} u  - \delta_k\rho\nabla\phi,
\end{equation}
where $\delta_{k-1}\chi(\cdot):=\chi(c_k)-\chi(c_{k-1})$ and
$\delta_{k-1}\kappa(\cdot):=\kappa(c_k)-\kappa(c_{k-1})$. Here zero
initial data are given, namely
$\delta_k\rho(\cdot,0)=0$, $\delta_k c(\cdot,0)=0$ and $\delta_k
u(\cdot,0)=0$.

For convenience, we denote
\[
h_1:=u_k \delta_k\rho, \qquad h_2:=\delta_{k-1} u
\rho_k,\qquad h_3:=\chi(c_{k}) \nabla c_{k+1} \delta_k\rho,
\]
\[
h_4:=\chi(c_k) \nabla\delta_{k} c \rho_k,
\qquad h_5:=\delta_{k-1}\chi(\cdot)
\nabla c_{k+1} \rho_k.
\]
Due to representation formula of  the heat equation, we get
\[
\delta_k\rho(x,t)=\int_0^t \bke{\nabla \Gamma * \sum_{i=1}^5 h_i}(s)ds.
\]
Using the estimate of the heat equation, we note that
\[
\norm{\delta_k\rho}_{L^{\infty}_tL^a_x}\le \sum_{i=1}^5\norm{\int_0^t \norm{\nabla \Gamma *  h_i(s)}_{L^a_x}ds}_{L^{\infty}_t}:=\sum_{i=1}^5H_i.
\]
We estimate $H_i$ separately. Indeed, for $H_1$ we obtain
\[
H_1=\norm{\int_0^t \norm{\nabla \Gamma *  h_1(s)}_{L^a_x}ds}_{L^{\infty}_t}\le C\norm{\int_0^t (t-s)^{-\frac{1}{2}}\norm{u_k(s)}_{L^{\infty}_x}\norm{\delta_k\rho}_{L^a_x}ds}_{L^{\infty}_t}
\]
\[
\le C\norm{u_k}_{l^2_tL^{\infty}_x}\norm{\delta_k\rho}_{L^{\infty}_tL^a_x}
\le Ct^{\eta}\norm{u_k}_{Y_a}\norm{\delta_k\rho}_{L^{\infty}_tL^a_x},
\]
where we used \eqref{Feb29-90-a}. Similarly, the second term $H_2$ can be estimated as follows:
\[
H_2=\norm{\int_0^t \norm{\nabla \Gamma *  h_2(s)}_{L^a_x}ds}_{L^{\infty}_t}
\le Ct^{\eta}\norm{\delta_{k-1} u}_{Y_a}\norm{\rho_k}_{L^{\infty}_tL^a_x}.
\]

To compute $H_3$, we introduce $1/q=1/(d+2)+1/a$ and we then obtain
\[
H_3
\le C\norm{\int_0^t (t-s)^{-\frac{d}{2}(\frac{1}{q}-\frac{1}{a})-\frac{1}{2}}\norm{\nabla c_k(s)}_{L^{d+2}_x}\norm{\delta_k\rho}_{L^a_x}ds}_{L^{\infty}_t}
\le C\norm{\nabla c_k}_{L^{d+2}_{x, t}}\norm{\delta_k\rho}_{L^{\infty}_tL^a_x}
\]
\[
\le C\bke{t^{1-\frac{d}{2a}}\norm{u_k}_{Y_a(Q_t)} +t^{\frac{3}{2}\bke{1-\frac{d}{2a}}}\norm{\rho}_{L^{\infty}_tL^{a}_x }
+t^{\frac{1}{2}\bke{1-\frac{d}{2a}}}\bke{\norm{u_0}_{L^{2a}_x }+\norm{\nabla c_0}_{L^{2a}_x }}}\norm{\delta_k\rho}_{L^{\infty}_tL^a_x}
\]
\[
\le C\bke{t^{\eta}+t^{\frac{3\eta}{2}}+t^{\frac{\eta}{2}}}\norm{\delta_k\rho}_{L^{\infty}_tL^a_x}.
\]
The term $H_5$ can be similarly estimated as $H_3$, namely
\[
H_5
\le C\norm{\int_0^t (t-s)^{-\frac{d}{2}(\frac{1}{q}-\frac{1}{a})-\frac{1}{2}}\norm{\delta_{k-1}\chi(\cdot)}_{L^{\infty}}
\norm{\nabla c_k}_{L^{d+2}_x}\norm{\rho_k}_{L^a_x}ds}_{L^{\infty}_t}
\]
\[
\le C\norm{\delta_{k-1}c}_{L^{\infty}}\norm{\nabla c_k}_{L^{d+2}_{x, t}}\norm{\rho_k}_{L^{\infty}_tL^a_x}
\]
\[
\le C\bke{t^{1-\frac{d}{2a}}\norm{u_k}_{Y_a(Q_t)} +t^{\frac{3}{2}\bke{1-\frac{d}{2a}}}\norm{\rho_k}_{L^{\infty}_tL^{a}_x }
+t^{\frac{1}{2}\bke{1-\frac{d}{2a}}}\bke{\norm{u_0}_{L^{2a}_x }+\norm{\nabla c_0}_{L^{2a}_x }}}\norm{\rho_k}_{L^{\infty}_tL^a_x}\norm{\delta_{k-1}c}_{L^{\infty}}
\]
\[
\le Ct^{\eta}\bke{t^{\eta}+t^{\frac{3\eta}{2}}+t^{\frac{\eta}{2}}}\norm{\delta_{k-1}c}_{Y_a}.
\]
Before we estimae $H_4$, we let $\tilde{a}$ with $1/\tilde{a}=1/a-1/d$. Choosing $\alpha$ and $\beta$ with
$\tilde{a}<\alpha<\infty$ and $d/\alpha+2/\beta=1$, we have
\[
\norm{\nabla \delta_k c}_{L^{\beta}_tL^{\alpha}_x}\le
C(\norm{\delta_{k-1} u c_{k+1}}_{L^{\beta}_tL^{\alpha}_x}+\norm{u_k\delta_{k}c}_{L^{\beta}_tL^{\alpha}_x}
+t^{1-\frac{d}{2a}}(\norm{\delta_{k-1}\kappa(\cdot) \rho_{k+1}}_{L^{\infty}_tL^a_x}+\norm{\kappa(c_k) \delta_k\rho}_{L^{\infty}_tL^a_x}))
\]
\[
\le C(\norm{\delta_{k-1} u }_{L^{\beta}_tL^{\alpha}_x}+\norm{u_k}_{L^{\beta}_tL^{\alpha}_x}\norm{\delta_{k}c}_{L^{\infty}_{x,t}}
+t^{1-\frac{d}{2a}}(\norm{\delta_{k-1}c}_{L^{\infty}_{x,t}}\norm{\rho_{k+1}}_{L^{\infty}_tL^a_x}+\norm{\delta_k\rho}_{L^{\infty}_tL^a_x}))
\]
\begin{equation}\label{18Nov-500}
\le Ct^{1-\frac{d}{2a}}(\norm{\delta_{k-1} u }_{Y_a}+\norm{\delta_{k}c}_{Y_a})
+Ct^{\frac{2a-d}{a}}\norm{\delta_{k-1}c}_{Y_a}+Ct^{1-\frac{d}{2a}}\norm{\delta_k\rho}_{L^{\infty}_tL^a_x}).
\end{equation}
Using the estimate \eqref{18Nov-500}, we obtain
\[
H_4
\le C\norm{\int_0^t (t-s)^{-\frac{d}{2\alpha}-\frac{1}{2}} \norm{\nabla\delta_{k} c}_{L^{\alpha}_x}\norm{\rho_k}_{L^a_x}ds}_{L^{\infty}_t}
\le C\norm{\nabla\delta_{k} c}_{L^{\beta}_tL^{\alpha}_x}\norm{\rho_k}_{L^{\infty}_tL^a_x}
\]
\[
\le Ct^{1-\frac{d}{2a}}(\norm{\delta_{k-1} u }_{Y_a}+\norm{\delta_{k}c}_{Y_a})
+Ct^{\frac{2a-d}{a}}\norm{\delta_{k-1}c}_{Y_a}+Ct^{1-\frac{d}{2a}}\norm{\delta_k\rho}_{L^{\infty}_tL^a_x})
\]
\[
\le Ct^{\eta}(\norm{\delta_{k-1} u }_{Y_a}+\norm{\delta_{k}c}_{Y_a})+t^{2\eta}\norm{\delta_{k-1}c}_{Y_a}
+Ct^{\eta_1}\norm{\delta_k\rho}_{L^{\infty}_tL^a_x}).
\]

Therefore, we obtain
\[
\norm{\delta_k\rho}_{X_a}
\lesssim \max(T^{\eta}, T^{\frac{\eta}{2}}, T^{\frac{3\eta}{2}})\norm{\delta_k\rho}_{X_a}
+T^{\eta}\norm{\delta_{k} c}_{Y_a}
\]
\begin{equation}\label{KK-Dec15-10}
+T^{\eta}\norm{\delta_{k-1} u}_{Y_a}+\max(T^{2\eta},
T^{\frac{5\eta}{2}}, T^{\frac{3\eta}{2}\eta})\norm{\delta_{k-1} c}_{Y_a}.
\end{equation}

Next, via maximal regularity of the Stokes system, we estimate $\delta u$ as follows:
\[
\norm{\partial_t \delta_k u}_{L^2_tL^a_x(Q_T)} +\norm{\Delta \delta_k
u}_{L^2_tL^a_x(Q_T)} +\norm{\nabla \delta_k p}_{L^2_tL^a_x(Q_T)}
\]
\[
\lesssim \norm{\delta_{k-1} u\nabla u_k}_{L^2_tL^a_x(Q_T)} +
\norm{u_{k-1}\nabla \delta_{k-1} u}_{L^2_tL^a_x(Q_T)}
+\norm{\delta_k\rho\nabla\phi}_{L^2_tL^a_x(Q_T)}
\]
\[
\lesssim \norm{\delta_{k-1} u}_{L^{\infty}}\norm{\nabla
u_k}_{L^2_tL^a_x(Q_T)}+\norm{u_{k-1}}_{L^{\infty}}\norm{\nabla \delta_{k-1}
u}_{L^2_tL^a_x(Q_T)}
+\norm{\nabla\phi}_{L^{\infty}}\norm{\delta_k\rho}_{L^2_tL^a_x(Q_T)}
\]
\[
\lesssim T^{\eta}\norm{\delta_{k-1}
u}_{Y_a}\norm{ u_{k}}_{Y_a}
+T^{\eta}\norm{\delta_{k-1}
u}_{Y_a}\bke{T^{\eta_1}\norm{ u_{k-1}}_{Y_a}+\norm{u_0}_{L^{\infty}}}+T^{\frac{1}{2}}\norm{\delta_k\rho}_{X_a}
\]
\[
\lesssim \bke{T^{2\eta}\norm{ u_{k-1}}_{Y_a}+T^{\eta}\norm{u_k}_{Y_a}+T^{\eta}\norm{u_0}_{L^{\infty}}}
\norm{\delta_{k-1}u}_{Y_a}+T^{\frac{1}{2}}\norm{\delta_k\rho}_{X_a}.
\]
Therefore,
\begin{equation}\label{KK-Dec15-20}
\norm{\delta_k u}_{Y_a}\lesssim \max(T^{2\eta}, T^{\eta})\norm{\delta_{k-1}
u}_{Y_a}+T^{\frac{1}{2}}\norm{\delta_k\rho}_{X_a}.
\end{equation}

Next, we estimate $\delta_k c$.
\[
\norm{\partial_t \delta_k c}_{L^2_tL^a_x(Q_T)}+\norm{\nabla^2 \delta_k
c}_{L^2_tL^a_x(Q_T)}
\]
\[
 \lesssim \norm{\delta_{k-1} u\nabla c_{k+1}}_{L^2_tL^a_x(Q_T)}+\norm{ u_k\nabla \delta_{k}c}_{L^2_tL^a_x(Q_T)}
 +\norm{\delta_{k-1}\kappa(\cdot) \rho_{k}}_{L^2_tL^a_x(Q_T)}+\norm{\kappa(c_k)
\delta_{k-1}\rho}_{L^2_tL^a_x(Q_T)}
\]
\[
\lesssim \norm{\delta_{k-1} u}_{L^{\infty}}\norm{\nabla
c_{k+1}}_{L^2_tL^a_x(Q_T)}+\norm{ u_k}_{L^{\infty}(Q_T)}\norm{\nabla
\delta_{k}c}_{L^2_tL^a_x(Q_T)}
\]
\[
+\norm{\delta_{k-1}\kappa(\cdot)}_{L^{\infty}}\norm{\rho_{k}}_{L^2_tL^a_x(Q_T)}+\norm{
\delta_{k-1}\rho}_{L^2_tL^a_x(Q_T)}
\]
\[
\lesssim T^{\eta}\norm{\delta_{k-1}
u}_{Y_a}\norm{c_{k+1}}_{Y_a}+T^{\frac{1}{2}}(T^{\eta}\norm{
u_k}_{Y_a}+\norm{
u_0}_{L^{\infty}})\norm{\delta_{k} c}_{Y_a}
\]
\[
+T^{\eta+\frac{1}{2}}\norm{\delta_{k-1}c}_{Y_a}\norm{\rho_{k}}_{X_a}
+T^{\frac{1}{2}}\norm{\delta_{k-1}\rho}_{X_a}
\]
\begin{equation*}
\lesssim T^{\eta}\norm{\delta_{k-1} u}_{Y_a}+\max(T^{\eta+\frac{1}{2}}, T^{\frac{1}{2}})
\norm{\delta_{k}c}_{Y_a}+
T^{\eta+\frac{1}{2}}\norm{\delta_{k-1}c}_{Y_a}
+T^{\frac{1}{2}}\norm{\delta_{k-1}\rho}_{X_a}.
\end{equation*}
Taking $T$ sufficiently small, we obtain
\begin{equation}\label{KK-May26-1000}
\norm{\delta_{k}c}_{Y_a}\lesssim T^{\eta}\norm{\delta_{k-1} u}_{Y_a}+
T^{\eta+\frac{1}{2}}\norm{\delta_{k-1}c}_{Y_a}
+T^{\frac{1}{2}}\norm{\delta_{k-1}\rho}_{X_a}.
\end{equation}

Summing up \eqref{KK-Dec15-10}, \eqref{KK-Dec15-20} and
\eqref{KK-May26-1000} and noting that \eqref{KK-Dec15-10} holds for all $d<a$, we have

\[
\norm{\delta_k\rho}_{X_a}+\norm{\delta_k
u}_{Y_a}+\norm{\delta_{k} c}_{Y_a}
\]
\[
\le C\max(T^{\eta}, T^{\frac{1}{2}})\norm{\delta_{k-1}\rho}_{X_a}
+\max(T^{2\eta}, T^{\eta+\frac{1}{2}}, T^{\eta})\norm{\delta_{k-1}
u}_{Y_a}
\]
\[
+C\max(T^{\eta+\frac{1}{2}}, T^{\eta},
T^{2\eta})\norm{\delta_{k-1} c}_{Y_a}.
\]
Taking $T$ sufficiently small, we obtain
\begin{equation}\label{KK-May27-100}
\norm{\delta_k\rho}_{X_a}+\norm{\delta_k
u}_{Y_a}+\norm{\delta_{k} c}_{Y_a}\le
\frac{1}{2}\bke{\norm{\delta_{k-1}\rho}_{X_a}+\norm{\delta_{k-1}
u}_{Y_a} +\norm{\delta_{k-1} c}_{Y_a}}.
\end{equation}
This yields fixed point via the theory of the contraction mapping.
Once we prove \eqref{KK-May27-300} and \eqref{KK-May27-400}, we
complete the proof. Note that this can be done exactly the same as
we did in the proof of Theorem 1.
\end{pfthm-2}


\section*{Acknowledgements}
Kyungkeun Kang's work is supported by NRF-2019R1A2C1084685 and
NRF-2015R1A5A1009350. Hwa Kil Kim's work is supported by NRF-2021R1F1A1048231 and NRF-2018R1D1A1B07049357.

\begin{equation*}
\left.
\begin{array}{cc}
{\mbox{Kyungkeun Kang}}\qquad&\qquad {\mbox{Hwa Kil Kim}}\\
{\mbox{Department of Mathematics }}\qquad&\qquad
 {\mbox{Department of Mathematics Education}} \\
{\mbox{Yonsei University
}}\qquad&\qquad{\mbox{Hannam University}}\\
{\mbox{Seoul, Republic of Korea}}\qquad&\qquad{\mbox{Daejeon, Republic of Korea}}\\
{\mbox{kkang@yonsei.ac.kr }}\qquad&\qquad {\mbox{hwakil@hnu.kr }}
\end{array}\right.
\end{equation*}


\begin{thebibliography}{00}

\bibitem{AKY} J. Ahn, K. Kang and C. Yoon,
 {\it Global classical solutions for chemotaxis-fluid systems in two dimensions}, Math. Methods Appl. Sci. 44(2); 2254-–2264, 2021.

\bibitem{ags:book}  L. Ambrosio, N. Gigli and  G. Savar{\'e}, {\it Gradient
  flows in metric spaces and in the space of probability  measures},
\newblock {\em Lectures in Mathematics ETH Z\"urich, Basel: Birkh\"auser Verlag,
  second~ed}, 2008.

%

\bibitem{ckl} M. Chae, K. Kang and J. Lee,
 {\it On Existence of the smooth solutions to the Coupled Chemotaxis-Fluid Equations}, Discrete Cont. Dyn. Syst. A, 33(6); 2271--2297, 2013.

\bibitem{ckl-jkms} M. Chae, K. Kang and J. Lee,
{\it Global existence and temporal decay in Keller-Segel models
coupled to fiuid equations},  Comm. Partial Diff. Equations, 39;
1205-1235, 2014.



\bibitem{ckl-cpde} M. Chae, K. Kang and J. Lee,
{\it Asymptotic behaviors of solutions for an aerobatic model
coupled to fluid equations},  J. Korean Math. Soc., 53 (1),
127--146, 2016.


\bibitem{ckl-nonlinearity} M. Chae, K. Kang and J. Lee,
{\it A regularity condition and temporal asymptotics for chemotaxis-fluid equations},  Nonlinearity, 31(2);
351–387, 2018.

\bibitem{CFKLM} A. Chertock, K. Fellner, A. Kurganov, A. Lorz and
P. A. Markowich, {\it Sinking, merging and stationary plumes in a
coupled chemotaxis-fluid model: a high-resolution numerical
approach}, J. Fluid Mech., 694; 155--190, 2012.


\bibitem{CKK} Y.-S. Chung, K. Kang and J. Kim, {\it
Global existence of weak solutions for a Keller-Segel-fluid model
with nonlinear diffusion}, J. Korean Math. Soc. 51(3); 635--654,
2014.

\bibitem{CK-jmp} Y.-S. Chung, K. Kang and J. Kim,
{\it Existence of global solutions for a chemotaxis-fluid system
with nonlinear diffusion}, J. Math. Phy., 57; 041503, 2016.


\bibitem{DLM} R. Duan, A. Lorz and P. Markowich, {\it Global
solutions to the coupled chemotaxis-fluid equations}, Comm. Partial
Diff. Equations, 35(9); 1635--1673, 2010.

\bibitem{FLM} M. D. Francesco, A. Lorz and P. Markowich, {\it
Chemotaxis-fluid coupled model for swimming bacteria with nonlinear
diffusion: global existence and asymptotic behavior}, Discrete
Cont. Dyn. Syst. A, 28(4); 1437-53, 2010.


%
%
%
%
%
%
%

\bibitem{GS} Y. Giga and H. Sohr, {\it
Abstract $ L^p$ estimates for the Cauchy problem with applications to the Navier–Stokes equations in exterior
domains}, J. Funct. Anal. 102 (1) 72–94,  (1991).

\bibitem{Her-Vela} M. A. Herrero and J. L. L. Velazquez, {\it A blow-up
mechanism for chemotaxis model}, Ann. Sc. Norm. Super. Pisa, 24(4);
633--683, 1997.


\bibitem{Horstmann1} D. Horstman, {\it From 1970 until present:
The Keller-Segel model in chemotaxis and its consequences I},
Jahresber. Deutsch. Math.-Verein. 105(3); 103-165, 2003.

\bibitem{Horstmann2} D. Horstman, {\it From 1970 until present:
The Keller-Segel model in chemotaxis and its consequences II},
Jahresber. Deutsch. Math.-Verein. 106(2); 51--69, 2004.

\bibitem{JL} W. J\H{a}ger and S. Luckhaus,
{\it On explosions of solutions to s system of partial differential
equations modeling chemotaxis}, Trans. AMS., 329(2), 819-824, 1992.

\bibitem{KK} K. Kang and H. K. Kim,
{\it Existence of weak solutions in Wasserstein space for a chemotaxis model coupled to fluid equations}, SIAM J. Math. Anal,
49(4), 2965-3004, 2017.




\bibitem{KS1} E. F. Keller and L. A. Segel, {\it Initiation of slide
mold aggregation viewd as an instability}, J. Theor. Biol., 26(3);
399--415, 1970.
\bibitem{KS2} E. F. Keller and L. A. Segel, {\it Model for
chemotaxis}, J. Theor. Biol., 30(2); 225--234, 1971.



%

\bibitem{KMS16} H. Kozono, M. Miura, and Y. Sugiyama, {\it Existence and uniqueness theorem on mild solutions to the Keller-Segel system coupled with the Navier-Stokes fluid}, J. Funct. Anal. 270(5),
1663–-1683, 2016.


\bibitem{LSU} O. A. Ladyzenskaja, V. A. Solonnikov and N. N. Ural'ceva,
    {\it Linear and Quasi-linear Equations of Parabolic Type},
     Amer. Math. Soc. (1968).

%

\bibitem{NSY} T. Nagai, T. Senba and K. Yoshida, {\it Applications
of the Trudinger-Moser inequality to a parabolic system of
chemotaxis}, Funkcial Ekvac. 40(3); 411--433, 1997.

%
%

\bibitem{OY} K. Osaki and A. Yagi, {\it Finite dimensional
attractors for one-dimensional Keller-Segel equations}, Funkcial
Ekvac. 44(3); 441--469, 2001.

\bibitem{Patlak} C. S.  Patlak, {\it Random walk with persistence and
external bias}, Bull. Math. Biol. Biophys. 15; 311--338, 1953.

%
%
%
%
%

\bibitem{S}  V. A. Solonnikov,
  {\it Estimates of solutions of the Stokes equations in S. L. Sobolev
  spaces with a mixed norm},
  Zap. Nauchn. Sem. S.-Peterburg. Otdel. Mat. Inst. Steklov. (POMI)
  {\bf 288}, (2002), pp. 204-231.

\bibitem{Tao-W} Y. Tao and M. Winkler, {\it Global existence and
boundedness in a Keller-Segel-Stokes model with arbitrary porous
medium diffusion}, Discrete
Cont. Dyn. Syst. A, 32(5);1901--1914, 2012.


\bibitem{TCDWKG} I. Tuval, L. Cisneros, C. Dombrowski, C. W.
Wolgemuth, J. O. Kessler and R. E. Goldstein, {\it Bacterial
swimming and oxygen transport near contact lines,} PNAS, 102(7);
2277--2282, 2005.

\bibitem{V} C. Villani, {\it Optimal transport. Old and new}. em Grundlehren der Mathematischen Wissenschaften [Fundamental Principles of
Mathematical Sciences], 338. Springer-Verlag, Berlin, 2009.

\bibitem{Win} M. Winkler, {\it Aggregation vs. global diffusive
behavior in the higher-dimensional Keller-Segel model}, J.
Differential Equations, 248(12); 2889--2995, 2010.

\bibitem{Win2} M. Winkler, {\it Global large data solutions in a
chemotaxis-(Navier-)Stokes system modeling cellular swimming in fluid drops}
Comm. Partial Diff. Equations, 37(2);319--351, 2012.

\bibitem{Win3} M. Winkler, {\it Stabilization in a two-dimensional
chemotaxis-Navier-Stokes system}, Arch. Ration. Mech. Anal,
211(2);455-487, 2014.

\end{thebibliography}
\end{document}